\documentclass[12pt]{amsart}

\usepackage[T1]{fontenc}
\usepackage{amsmath,amssymb,amsthm,mathtools}
\usepackage{amsfonts}
\usepackage{mathrsfs}
\usepackage[top=1.3in,bottom=1.2in,left=1in,right=1in]{geometry}
\usepackage[activate=false]{microtype}
\usepackage[english]{babel}
\usepackage[colorlinks=true,linkcolor=blue,citecolor=blue,urlcolor=blue,backref=page]{hyperref}
\renewcommand*{\backref}[1]{}
\renewcommand*{\backrefalt}[4]{\ifcase #1\or $\uparrow$ #2\else $\uparrow$ #2\fi}
\usepackage{enumitem}
\usepackage[nameinlink,capitalize]{cleveref}

\usepackage{longtable,booktabs,array}

\numberwithin{equation}{section}

\newtheorem{theorem}{Theorem}[section]
\newtheorem{proposition}[theorem]{Proposition}
\newtheorem{corollary}[theorem]{Corollary}
\newtheorem{lemma}[theorem]{Lemma}
\newtheorem{definition}[theorem]{Definition}
\newtheorem{remark}{Remark}[section]

\newcommand{\RR}{\mathbb R}
\newcommand{\CC}{\mathbb C}
\newcommand{\HH}{\mathbb H}

\newcommand{\CP}{\mathbb{CP}}
\newcommand{\HP}{\mathbb{HP}}

\newcommand{\Id}{\operatorname{Id}}
\newcommand{\tr}{\operatorname{tr}}

\newcommand{\Span}{\operatorname{span}}

\newcommand{\calQ}{\mathcal Q}
\newcommand{\inner}[2]{\left\langle #1,#2\right\rangle}

\newcommand{\Sym}{\operatorname{Sym}}
\newcommand{\Herm}{\operatorname{Herm}}
\newcommand{\Spin}{\mathrm{Spin}}
\newcommand{\Vol}{\operatorname{Vol}}

\def\End{\mathrm{End}}

\newcommand{\sph}{\mathrm S}

\setlist[enumerate]{itemsep=0.3em}
\begin{document}

\title[Normal curvature and Veronese rigidity]
{Maximal Normal Curvature and Veronese Rigidity}

\author{Tsz-Kiu Aaron Chow}
\address{Department of Mathematics, Hong Kong University of Science and Technology, Hong Kong S.A.R., China}
\email{\href{chowtka@ust.hk}{chowtka@ust.hk}}
\thanks{T.-K. A. C. is supported  by the Croucher Foundation Start-up Grant and the HKUST New Faculty Start-up Grant}

\author{Jingbo Wan}
\address{Laboratoire Jacques-Louis Lions de Sorbonne Universit\'e, 4 place Jussieu, Paris 75005, France}
\email{\href{jingbo.wan@sorbonne-universite.fr}{jingbo.wan@sorbonne-universite.fr}}
\thanks{J. W. is supported by ERC-2023 AdG 101141855 BLaHST}

\begin{abstract}
We prove a sharp Veronese rigidity theorem for closed immersed submanifolds of the
Euclidean unit ball under intrinsic harmonic-structure assumptions.  For an
isometric immersion $F:(\Sigma,g)\looparrowright\overline B(1)$, define the maximal normal curvature by
\[
        \kappa(F):=
        \sup_{x\in\Sigma}
        \sup_{\substack{v\in T_x\Sigma\\ |v|_g=1}}
        |A_x(v,v)|.
\]
If $\Sigma^{2n}$ is almost Hermitian with harmonic fundamental two-form, or
$\Sigma^{4n}$ is almost quaternion-Hermitian with harmonic fundamental
four-form, $n\ge2$, then
\[
        \kappa(F)\ge \sqrt{\frac{2n}{n+1}} .
\]
In the equality case the harmonic form is parallel and the immersion is,
up to a totally geodesic inclusion, the standard complex or quaternionic
Veronese embedding of projective spaces.  The key input is a Bochner--Gauss mechanism that turns the
Bochner curvature term of the harmonic form into a sharp algebraic estimate for
the shape operators.
\end{abstract}
\maketitle
\tableofcontents
%============================================================
\section{Introduction}
\label{sec:introduction}
%============================================================

A theme initiated by Gromov is to study isometric immersions into Euclidean
balls under quantitative control of their extrinsic curvature
\cite{GromovControlledCurvatures}.  Let $(\Sigma^m,g)$ be closed and connected,
and let
\[
        F:(\Sigma^m,g)\looparrowright
        \overline B^{m+\ell}(1)\subset\RR^{m+\ell}
\]
be an isometric immersion.  We write $D$ for the Euclidean connection and
$\nabla$ for the Levi-Civita connection of $(\Sigma,g)$.  For tangent vector
fields $X,Y$ on $\Sigma$, viewed as vector fields along the immersion, the
second fundamental form is
\[
        A(X,Y)=D_XY-(D_XY)^T=(D_XY)^\perp .
\]
The position vector decomposes as $F=F^T+F^\perp$.  Following Gromov \cite{GromovControlledCurvatures}, the
maximal normal curvature of the immersion is
\begin{align}
        \kappa(F)
        :=
        \sup_{x\in\Sigma}\sup_{\substack{v\in T_x\Sigma\\ |v|=1}}
        |A_x(v,v)| .
        \label{eq:intro-kappa-definition}
\end{align}
Equivalently, $\kappa(F)$ is the maximal Euclidean curvature of intrinsic
unit-speed geodesics in $\Sigma$.

Petrunin's Veronese theorem shows that this quantity detects very rigid
borderline geometry: If a closed submanifold of an $r$-ball has $\kappa(F)\le \frac{2}{\sqrt3\,r}$, then the strict inequality forces the source to be homeomorphic to a sphere,
while any non-spherical borderline case is, up to rescaling and congruence, one
of the Veronese embeddings of projective planes $\mathbb{FP}^2$ \cite{PetruninVeronese}.  A closely related
strand of the same program fixes the smooth source manifold $X$ and asks for
the optimal value of $\kappa(F)$ among immersions of $X$ into Euclidean balls
of sufficiently large dimension.  In this fixed-topology direction, some sharp
values are known, for instance, for $X=T^n$ \cite{PetruninGromovTori} and for
$X=S^n\times S^1$ \cite{ChodoshLiSmallNormalCurvature}.
These results identify optimal lower bounds for $\kappa(F)$; however, the
corresponding equality theory is subtler and does not, in general, give a classification of all minimizers.

The present paper takes a complementary viewpoint to both of these directions.  We do not prescribe the
topology of $\Sigma$.  Instead, we prescribe an intrinsic special geometric
structure: an almost Hermitian structure with harmonic fundamental two-form, or
an almost quaternion-Hermitian structure with harmonic fundamental four-form.
The resulting variational problem is to minimize the maximal normal curvature
among all closed isometric immersions into the unit ball carrying such a
harmonic structure form.  Our result identifies the sharp minimum with the
Veronese normal-curvature scale.  Moreover, equality is fully rigid: a minimizer
automatically has the projective topology, the intrinsic projective metric, and
the standard Veronese embedding of $\mathbb{CP}^n$ and $\mathbb{HP}^n$ for all $n\geq 2$.

\begin{definition}\
\begin{enumerate}[label=(\roman*)]
\item An almost Hermitian manifold is a triple $(\Sigma^{2n},g,J)$, where
$J^2=-\Id$ and $g(JX,JY)=g(X,Y)$.  Its fundamental two-form is
\[
        \omega(X,Y)=\langle JX,Y\rangle .
\]
If $\nabla\omega=0$, then $(\Sigma,g,J)$ is K\"ahler.

\item An almost quaternion-Hermitian manifold is a triple $(\Sigma^{4n},g,\calQ)$,
where $\calQ\subset\End(T\Sigma)$ is a rank-three subbundle locally spanned by
endomorphisms $I,J,K$ satisfying
\[
I^2=J^2=K^2=-\Id,
\qquad
IJK=-\Id,
\]
and such that $g(\phi X,\phi Y)=g(X,Y)$ for every local section $\phi\in\calQ$ with
$\phi^2=-\Id$.  For a local admissible frame $I,J,K$, set
\[
        \omega_I(X,Y)=\langle IX,Y\rangle,
        \qquad
        \omega_J(X,Y)=\langle JX,Y\rangle,
        \qquad
        \omega_K(X,Y)=\langle KX,Y\rangle .
\]
The fundamental four-form is
\[
        \Theta
        =
        \frac16(\omega_I^2+\omega_J^2+\omega_K^2).
\]
It is independent of the choice of local admissible frame.  If
$\nabla\Theta=0$, then $(\Sigma,g,\calQ)$ is quaternionic-K\"ahler.
\end{enumerate}
\end{definition}

\begin{theorem}
\label{thm:main-ball-FPn}
Let $(\Sigma^m,g)$ be closed and connected, and let $F:(\Sigma^m,g)\looparrowright
        \overline B^{m+\ell}(1)\subset\RR^{m+\ell}$ be an isometric immersion. \footnote{For a fixed $\ell$, such an isometric immersion may not exist. By the smooth Nash embedding theorem \cite{Nash1956}, after a constant rescaling of $g$, such an isometric immersion into $\overline B^{m+\ell}(1)$ exists for sufficiently large $\ell$.}
\begin{enumerate}[label=(\roman*)]
\item If $(\Sigma^{2n},g,J)$ is almost Hermitian, its fundamental two-form
$\omega$ is harmonic, and $n\ge2$, then
\[
        \kappa(F)\ge \sqrt{\frac{2n}{n+1}} .
\]
Equality holds if and only if $(\Sigma,g,J)$ is K\"ahler,
$F(\Sigma)\subset\partial B(1)$, and $F$ is globally congruent to the
Veronese embedding $\CP^n\subset S^{(n+1)^2-2}(1)$, up to a totally geodesic inclusion.

\item If $(\Sigma^{4n},g,\calQ)$ is almost quaternion-Hermitian, its
fundamental four-form $\Theta$ is harmonic, and $n\ge2$, then
\[
        \kappa(F)\ge \sqrt{\frac{2n}{n+1}} .
\]
Equality holds if and only if $(\Sigma,g,\calQ)$ is quaternionic-K\"ahler,
$F(\Sigma)\subset\partial B(1)$, and $F$ is globally congruent to the
Veronese embedding $\HP^n\subset S^{(n+1)(2n+1)-2}(1)$, up to a totally geodesic inclusion.
\end{enumerate}
\end{theorem}

\begin{remark}
Only the estimate uses the harmonic assumptions.  In the equality case, the
Bochner identity gives $\nabla\omega=0$, respectively $\nabla\Theta=0$, so the
rigidity part reduces to the K\"ahler and quaternionic-K\"ahler cases.

For $n\ge3$, closedness of $\Theta$ already implies $\nabla\Theta=0$ by
Swann's theorem \cite{Swann1991}.  Thus the almost quaternion-Hermitian
formulation is only a genuine weakening in real dimension $8$.
\end{remark}

\begin{remark}[The octonionic plane]
There is a natural octonionic analogue only for the Cayley projective plane $\mathbb OP^2=F_4/\Spin(9)$, not for a family $\mathbb OP^n$, $n\ge3$, because of the non-associativity of $\mathbb O$.  A $16$-manifold with a $\Spin(9)$-structure carries a canonical invariant $8$-form, and the corresponding assumption would be its harmonicity.  We expect the same sharp estimate and rigidity statement for the octonionic Veronese embedding of $\mathbb OP^2$ as a single exceptional case, but the required Bochner--Gauss algebra for the $\Spin(9)$-invariant $8$-form is more involved than in the Hermitian and
quaternionic-Hermitian cases.
\end{remark}

\begin{corollary}
\label{cor:variational-form}
For $n\ge2$, define
\[
\mathcal P_{\CC,n}
=
\left\{
(\Sigma^{2n},g,J,F,\ell)
\ \middle|\
\begin{aligned}
&\ell\ge1,\quad
(\Sigma,g,J)\ \text{closed connected almost Hermitian},\\
&\omega\ \text{harmonic},\quad
F:(\Sigma,g)\looparrowright \overline B^{2n+\ell}(1)
\ \text{is an isometric immersion}
\end{aligned}
\right\}.
\]
and
\[
\mathcal P_{\HH,n}
=
\left\{
(\Sigma^{4n},g,\calQ,F,\ell)
\ \middle|\
\begin{aligned}
&\ell\ge1,\quad
(\Sigma,g,\calQ)\ \text{closed connected almost quaternion-Hermitian},\\
&\Theta\ \text{harmonic},\quad
F:(\Sigma,g)\looparrowright \overline B^{4n+\ell}(1)
\ \text{is an isometric immersion}
\end{aligned}
\right\}.
\]
Then
\[
        \inf_{(\Sigma,g,J,F,\ell)\in\mathcal P_{\CC,n}}\kappa(F)
        =
        \sqrt{\frac{2n}{n+1}},
        \qquad
        \inf_{(\Sigma,g,\calQ,F,\ell)\in\mathcal P_{\HH,n}}\kappa(F)
        =
        \sqrt{\frac{2n}{n+1}} .
\]
In both cases, the minimizers are exactly the standard complex,
respectively quaternionic, Veronese embeddings, up to reparametrization,
ambient orthogonal congruence, and totally geodesic inclusion of the target
sphere.
\end{corollary}

\begin{proof}
For each fixed $\ell$, Theorem~\ref{thm:main-ball-FPn} gives
\[
        \kappa(F)\ge \sqrt{\frac{2n}{n+1}} .
\]
Taking the infimum over all finite $\ell$ gives the same lower bound on
$\mathcal P_{\CC,n}$ and $\mathcal P_{\HH,n}$.

There is no loss of compactness in the codimension parameter for the value of
the infimum: the standard Veronese embeddings $\Phi_{\mathbb C}$  and $\Phi_{\mathbb H}$ \eqref{eq:model-FPn-map} for $\mathbb{F}=\mathbb{C},\mathbb{H}$  already occur in finite
codimension and realize
\[
        \kappa(\Phi_{\mathbb C})=\sqrt{\frac{2n}{n+1}},\qquad   \kappa(\Phi_{\mathbb H})=\sqrt{\frac{2n}{n+1}},
\]
see already Section \ref{sec:models}. Thus the lower bound is sharp without passing to a sequence with
$\ell\to\infty$.  Composing the standard models with totally geodesic
inclusions of spheres gives the same value in larger codimensions.  Hence both
infima have the stated value.  The classes $\mathcal P_{\CC,n}$ and
$\mathcal P_{\HH,n}$ are non-empty; more generally, by the smooth Nash
embedding theorem \cite{Nash1956}, after a constant rescaling of the metric,
any closed admissible source admits an isometric immersion into a Euclidean
ball of sufficiently large dimension.

If equality holds for some admissible $F$, then $F$ has some fixed finite
codimension.  The equality statement of Theorem~\ref{thm:main-ball-FPn}
therefore applies and gives the stated classification of minimizers.
\end{proof}

\subsection{The Bochner--Gauss mechanism}

The proof is based on an extrinsic use of the Bochner formula.  In its
classical form, the Bochner method starts from an intrinsic curvature
positivity assumption and concludes that harmonic forms, tensors, or spinors
vanish or become parallel.  Here the direction is different: the geometric
structure on $\Sigma$ supplies a distinguished harmonic form, and this form
is used to test the curvature produced by the immersion.  More precisely, the
Weitzenb\"ock curvature term is split by the Euclidean Gauss equation into the
part determined by the second fundamental form.  This extrinsic
Bochner--Gauss viewpoint was used in our previous work on focal-radius
rigidity \cite{ChowWanFocalHypersurfaces}; related formulae for the Bochner
term of an isometric immersion also appear in Savo's work
\cite{Savo2014}.

More precisely, applying the Bochner formula to the complex or quaternionic
fundamental form gives an integral identity involving its Bochner curvature
term.  The Euclidean Gauss equation then rewrites this term entirely in terms
of the shape operators of the immersion.  The problem is thereby reduced to a
sharp pointwise algebraic inequality for the action of symmetric endomorphisms
on the relevant structure form.  This converts the harmonicity assumption into
a lower bound for the maximal normal curvature.

The same mechanism also controls the equality case.  Equality in the
Bochner--Gauss inequality forces equality in the underlying algebraic estimate,
which imposes a rigid Hermitian or quaternionic-Hermitian structure on the
shape operators.  Combined with the extrinsic Minkowski identity, this yields
spherical minimality and ultimately identifies the immersion with the standard
complex or quaternionic Veronese model.

Let
\[
        (m,p,\Psi,\mathbb F)=(2n,2,\omega,\CC)
        \quad\text{or}\quad
        (m,p,\Psi,\mathbb F)=(4n,4,\Theta,\HH).
\]
Since $\Psi$ is harmonic, the Bochner formula and the Euclidean Gauss equation
give the integral identity \cite{ChowWanFocalHypersurfaces, Savo2014}
\begin{equation}\label{eq:intro-bochner-gauss}
        0=
        \int_\Sigma |\nabla\Psi|^2\,d\mu
        +
        \int_\Sigma
        \left\langle
        \sum_\alpha q^{(p)}_{A_\alpha}\Psi,\Psi
        \right\rangle d\mu .
\end{equation}
Here $A_\alpha$ are the scalar shape operators in an orthonormal frame of the
normal bundle, and $q^{(p)}_{A_\alpha}$ denotes the curvature endomorphism on
$p$-forms associated with the Gauss term generated by $A_\alpha$.  A priori the
second term in \eqref{eq:intro-bochner-gauss} has no sign.  The special form
$\Psi$ changes this: it selects a representation-theoretically distinguished
part of the Gauss curvature, and the relevant algebra forces the correct sign
after comparison with the maximal normal curvature.

The pointwise heart of the paper is the Bochner algebra developed in Section \ref{sec:linear-algebra}.  In the complex case a symmetric endomorphism is split into its Hermitian trace-free, scalar, and anti-Hermitian parts.  In the quaternionic case one uses the decomposition
\[
        \Sym_\RR(\HH^n)
        =
        \Herm_0(\HH^n)
        \oplus \RR\Id
        \oplus S_I\oplus S_J\oplus S_K .
\]
For both structures this gives the sharp pointwise estimate
\begin{equation}\label{eq:intro-pointwise-algebra}
        -\frac{\left\langle
        \sum_\alpha q^{(p)}_{S_\alpha}\Psi,\Psi
        \right\rangle}
        {p(m-p)|\Psi|^2}
        \le
        \frac{n+1}{n-1}
        \sup_{|v|=1}\sum_\alpha \langle S_\alpha v,v\rangle^2
        -
        \frac{2n}{n-1}\sum_\alpha
        \left(\frac{\tr S_\alpha}{m}\right)^2,
\end{equation}
with equality conditions.  Applying this to $S_\alpha=A_\alpha$ and using
\eqref{eq:intro-bochner-gauss} yields
\begin{equation}\label{eq:intro-main-integral-estimate}
        \int_\Sigma \sup_{|v|=1}|A(v,v)|^2\,d\mu
        \ge
        \frac{2n}{n+1}
        \int_\Sigma \frac{1}{m^2}|\tr A|^2\,d\mu .
\end{equation}
The remaining step is not a curvature argument but the Euclidean normalization.
The Minkowski identity
\[
        \int_\Sigma \langle \tr A,F^\perp\rangle\,d\mu
        =
        -m\Vol(\Sigma)
\]
combined with $|F|\le1$ gives
\[
        \int_\Sigma \frac{1}{m^2}|\tr A|^2\,d\mu
        \ge
        \Vol(\Sigma).
\]
Together with \eqref{eq:intro-main-integral-estimate}, this proves
$\kappa(F)^2\ge 2n/(n+1)$.

This mechanism is local and algebraic at its core.  The harmonic form is used
as an extrinsic probe: the Bochner formula detects the part of the Gauss
curvature created by the second fundamental form, and the special-geometry
algebra converts that detection into a sharp quantitative estimate.  We expect
that analogous Bochner--Gauss algebras can be useful for other special
geometric structures and other extrinsic rigidity problems.

\subsection{Rigidity from the equality algebra}
The equality case uses more information from Bochner than the usual conclusion
that a harmonic object is parallel.  Equality in \eqref{eq:intro-bochner-gauss}
indeed gives
\[
        \nabla\omega=0
        \quad\text{or}\quad
        \nabla\Theta=0,
\]
so the almost Hermitian or almost quaternion-Hermitian structure becomes
K\"ahler or quaternionic-K\"ahler.  Equality in the Minkowski step gives
$F(\Sigma)\subset S(1)$ and $\tr A=-mF$, hence the spherical immersion is
minimal.

The decisive extra input comes from equality in the algebraic estimate
\eqref{eq:intro-pointwise-algebra}.  It forces every spherical shape operator
to be trace-free Hermitian and gives, for every unit tangent vector $v$,
\[
        |A^{\sph}(v,v)|^2=\frac{n-1}{n+1}.
\]
The Gauss equation then identifies the intrinsic metric as the positive complex
or quaternionic space form with holomorphic, respectively quaternionic,
sectional curvature $2n/(n+1)$.  The Codazzi equation, together with a simple
Hermitian linear-algebra lemma, implies that the first normal bundle is
parallel and that the spherical second fundamental form is parallel.

At this point the equality algebra continues to control the model.  The
shape-operator map
\[
        E\longrightarrow \Herm_0(T\Sigma;\mathbb F),
        \qquad
        \eta\longmapsto A^{\sph}_\eta,
\]
from the first normal bundle to the trace-free $\mathbb F$-Hermitian
endomorphisms is a parallel homothety.  After reducing codimension, one
compares the immersion on the universal cover with the standard Veronese model
through this parallel homothety and applies the fundamental theorem of
submanifolds in the sphere.  The Veronese map is injective, so the comparison
descends from the universal cover and gives the global congruence.

Thus the Bochner formula supplies both the estimate and the beginning of the
classification.  The equality case of the curvature-endomorphism algebra is
strong enough to recover the full projective-space embedding geometry, not
merely a parallel intrinsic structure form.  This is the rigidity feature that
makes the method different from the classical Bochner paradigm.

\subsection{Organization}
In Section \ref{sec:models} we record the complex and quaternionic Veronese models in
the Euclidean unit ball and verify the sharp value of $\kappa$.  In
Section \ref{sec:linear-algebra} we prove the pointwise Bochner algebra estimates for
the K\"ahler form and the quaternionic fundamental four-form, including their
equality cases.  Section \ref{sec:proof-estimate-main} proves the sharp lower bound by combining
these estimates with the Bochner formula, the Gauss equation, and the Minkowski
identity.  Section \ref{sec:rigidity} proves the equality case and identifies the
immersion with the standard Veronese embedding.

%============================================================
\section{The standard projective models in the Euclidean ball}
\label{sec:models}
%============================================================

We record the sharp examples as isometric immersions of the standard projective
spaces.  Let $\mathbb F=\CC$ or $\HH$, and equip $\mathbb FP^n$ with the
standard projective metric $g_{\mathrm{FS}}$ normalized as follows.  At
$[e_0]\in\mathbb FP^n$, write
\[
        \mathbb F^{n+1}=\mathbb Fe_0\oplus\mathbb F^n .
\]
A tangent vector is represented by $v\in\mathbb F^n$, and
\[
        |v|_{g_{\mathrm{FS}}}^2
        =
        \frac{2(n+1)}{n}|v|^2 .
\]
This normalization is invariant under the projective isometry group.

For a matrix $C$ over $\mathbb F$, set
\[
        \tr_\mathbb F C
        =
        \begin{cases}
        \tr_\CC C, & \mathbb F=\CC,\\
        \Re\sum_i C_{ii}, & \mathbb F=\HH.
        \end{cases}
\]
Let
\[
        \mathcal H_\mathbb F
        =
        \{C\in \mathrm{Mat}_{n+1}(\mathbb F): C^*=C,\ \tr_\mathbb F C=0\},
        \qquad
        \inner{C_1}{C_2}=\tr_\mathbb F(C_1C_2).
\]
Then
\[
        \dim_\RR\mathcal H_\CC=(n+1)^2-1,
        \qquad
        \dim_\RR\mathcal H_\HH=(n+1)(2n+1)-1 .
\]

View $\mathbb FP^n$ as the space of right $\mathbb F$-lines.  Define
\begin{equation}\label{eq:model-FPn-map}
        \Phi_\mathbb F:(\mathbb FP^n,g_{\mathrm{FS}})
        \longrightarrow S(\mathcal H_\mathbb F),
        \qquad
        \Phi_\mathbb F([z])
        =
        \sqrt{\frac{n+1}{n}}
        \left(zz^*-\frac{1}{n+1}I_{n+1}\right),
        \qquad
        |z|=1 .
\end{equation}
This is well-defined because $(zq)(zq)^*=zz^*$ for $|q|=1$.  If $P=zz^*$,
then $P^2=P$, $\tr_\mathbb F P=1$, and
\[
        \left|P-\frac{1}{n+1}I_{n+1}\right|^2
        =
        1-\frac{2}{n+1}+\frac{1}{n+1}
        =
        \frac{n}{n+1}.
\]
Thus $\Phi_\mathbb F$ maps into the unit sphere of $\mathcal H_\mathbb F$.

We verify the isometry at $[e_0]$.  For a tangent vector represented by
$v\in\mathbb F^n$, take
\[
        z(t)=\cos(|v|t)e_0+\sin(|v|t)\frac{v}{|v|},
        \qquad
        P(t)=z(t)z(t)^* .
\]
Then
\[
        P'(0)=
        \begin{pmatrix}
        0&v^*\\
        v&0
        \end{pmatrix},
        \qquad
        P''(0)=
        \begin{pmatrix}
        -2|v|^2&0\\
        0&2vv^*
        \end{pmatrix}.
\]
Hence
\[
        |d\Phi_\mathbb F(v)|^2
        =
        \frac{n+1}{n}|P'(0)|^2
        =
        \frac{2(n+1)}{n}|v|^2
        =
        |v|_{g_{\mathrm{FS}}}^2 .
\]
Therefore $\Phi_\mathbb F$ is an isometric immersion.  In particular, a
$g_{\mathrm{FS}}$-unit tangent vector is represented by
$|v|^2=n/(2(n+1))$.

Let $A^S$ be the second fundamental form of $\Phi_\mathbb F$ in the unit
sphere and $A$ the second fundamental form in the Euclidean space
$\mathcal H_\mathbb F$.  Since $\Phi_\mathbb F$ has unit length,
\begin{equation}\label{eq:model-Euclidean-spherical-A}
        A(X,Y)=A^S(X,Y)-\inner{X}{Y}\Phi_\mathbb F .
\end{equation}
Hence, for $|u|_{g_{\mathrm{FS}}}=1$,
\begin{equation}\label{eq:model-Euclidean-spherical-norm}
        |A(u,u)|^2=|A^S(u,u)|^2+1 .
\end{equation}

We compute $A^S$ at $[e_0]$.  The spherical normal space is
\[
        N^S_{[e_0]}\mathbb FP^n
        =
        \left\{
        \begin{pmatrix}
        0&0\\
        0&B
        \end{pmatrix}:
        B^*=B,\ \tr_\mathbb F B=0
        \right\}.
\]
For a $g_{\mathrm{FS}}$-unit tangent vector represented by
$v=\sqrt{n/(2(n+1))}\,u$, with $|u|=1$, differentiating
\eqref{eq:model-FPn-map} twice and adding the radial term gives
\begin{equation}\label{eq:model-FPn-spherical-A}
        A^S(u,u)
        =
        \sqrt{\frac{n}{n+1}}
        \begin{pmatrix}
        0&0\\
        0&uu^*-\frac{1}{n}I_n
        \end{pmatrix}.
\end{equation}
Consequently,
\begin{align}
        |A^S(u,u)|^2
        &=
        \frac{n}{n+1}
        \left|uu^*-\frac{1}{n}I_n\right|^2                       \nonumber\\
        &=
        \frac{n}{n+1}
        \left(
        1-\frac{2}{n}+\frac{1}{n}
        \right)
        =
        \frac{n-1}{n+1}.
\label{eq:model-FPn-spherical-sharp}
\end{align}
By \eqref{eq:model-Euclidean-spherical-norm},
\[
        \kappa(\Phi_\mathbb F)^2
        =
        1+\frac{n-1}{n+1}
        =
        \frac{2n}{n+1}.
\]
Thus the constant in Theorem \ref{thm:main-ball-FPn} is sharp for both
$\CP^n$ and $\HP^n$.

%------------------------------------------------------------
\subsection{A second isometric immersion of the standard complex projective plane}
\label{subsec:model-nonunique-CP2}
%------------------------------------------------------------

The intrinsic metric alone does not determine the isometric immersion.  Let
\[
        V=\Phi_\CC:(\CP^2,g_{\mathrm{FS}})\longrightarrow S^7(1)\subset\RR^8
\]
be \eqref{eq:model-FPn-map}.  Then $V$ is isometric and
\[
        \kappa(V)^2=\frac43 .
\]
Set
\[
        F_0:\CP^2\longrightarrow \RR^{35},
        \qquad
        F_0=(V,0)\in\RR^8\oplus\RR^{27}.
\]
Then $F_0$ is isometric, $F_0(\CP^2)\subset S^7(1)\subset\overline B^{35}(1)$,
and
\begin{equation}\label{eq:model-F0-kappa}
        \kappa(F_0)^2=\frac43 .
\end{equation}

Define
\[
        \mathcal V:S^7(1)\longrightarrow \Sym_0(\RR^8)\simeq\RR^{35},
        \qquad
        \mathcal V(x)
        =
        \frac1{\sqrt2}
        \left(
        xx^T-\frac18I_8
        \right).
\]
For $v\in T_xS^7(1)$, so $\inner{x}{v}=0$,
\[
        d\mathcal V_x(v)
        =
        \frac1{\sqrt2}(xv^T+vx^T),
        \qquad
        |d\mathcal V_x(v)|^2
        =
        \frac12\tr (xv^T+vx^T)^2
        =
        |v|^2 .
\]
Thus $\mathcal V$ is an isometric immersion.  Also
\[
        |\mathcal V(x)|^2
        =
        \frac12\tr\left(xx^T-\frac18I_8\right)^2
        =
        \frac12
        \left(
        1-\frac14+\frac18
        \right)
        =
        \frac7{16}.
\]
Hence
\[
        F_1:=\mathcal V\circ V:(\CP^2,g_{\mathrm{FS}})
        \longrightarrow \RR^{35}
\]
is an isometric immersion and
\[
        F_1(\CP^2)\subset S^{34}\!\left(\frac{\sqrt7}{4}\right)
        \subset B^{35}(1).
\]

We compute $\kappa(F_1)$.  Along a unit-speed geodesic
$x(t)=\cos t\,x+\sin t\,v$ in $S^7(1)$,
\[
        \frac{d^2}{dt^2}\mathcal V(x(t))\bigg|_{t=0}
        =
        \sqrt2(vv^T-xx^T),
        \qquad
        \left|\sqrt2(vv^T-xx^T)\right|^2=4 .
\]
Thus
\[
        |A_{\mathcal V}(v,v)|^2=4
        \qquad
        (|v|=1).
\]
For a $g_{\mathrm{FS}}$-unit vector $u\in T\CP^2$,
\[
        A_{F_1}(u,u)
        =
        d\mathcal V(A^S_V(u,u))
        +
        A_{\mathcal V}(dV(u),dV(u)),
\]
and the two terms are orthogonal.  Hence, by
\eqref{eq:model-FPn-spherical-sharp} with $n=2$,
\begin{equation}\label{eq:model-F1-kappa}
        |A_{F_1}(u,u)|^2
        =
        |A^S_V(u,u)|^2+4
        =
        \frac13+4
        =
        \frac{13}{3}.
\end{equation}
Therefore
\[
        \kappa(F_1)^2=\frac{13}{3}\neq \kappa(F_0)^2.
\]
By \eqref{eq:model-F0-kappa} and \eqref{eq:model-F1-kappa}, the isometric
immersions $F_0$ and $F_1$ are not congruent.
\section{The Bochner algebra of the fundamental forms}
\label{sec:linear-algebra}

This section is purely pointwise.  We use only the algebraic Hermitian,
respectively quaternion-Hermitian, structure on a Euclidean vector space.  No
closedness, harmonicity, or parallelism of the fundamental forms is used. Let $V$ be a real Euclidean vector space of dimension $m$.  For
$S\in\Sym(V)$, define $S^{[p]}$ on $\Lambda^pV^*$ as follows.  If
$SE_a=\mu_aE_a$ in an orthonormal basis, then for $E^A = E^{a_1}\wedge\cdots\wedge E^{a_r}$,
\begin{align*}
        S^{[p]}E^A
        =
        \left(\sum_{a\in A}\mu_a\right)E^A .
\end{align*}
We use the convention
\begin{align*}
        q_S^{(p)}E^A
        =
        \left(\sum_{a\in A}\mu_a\right)
        \left(\tr_\RR S-\sum_{a\in A}\mu_a\right)E^A .
\end{align*}
Here $q_S^{(p)}$ is the curvature endomorphism appearing in the Bochner--Gauss formula \eqref{eq:intro-bochner-gauss}.
All sphere measures below are normalized.

%------------------------------------------------------------
\subsection{The complex projective case}
\label{subsec:complex-projective-algebra}
%------------------------------------------------------------

Let $V=\CC^n$, viewed as a real Euclidean vector space of dimension $2n$, and
let $J$ be multiplication by $i$.  We use $\inner{X}{Y}= \Re\sum_{\alpha=1}^n X_\alpha\overline{Y_\alpha}$, where the bar denotes the standard complex conjugation.  Set $\omega(X,Y)=\inner{JX}{Y}$.
For $S\in\Sym_\RR(\CC^n)$, write
\[
        S=H+\tau\Id+K,
        \qquad
        \tau=\frac{1}{2n}\tr_\RR S,
\]
where $HJ=JH$, $H^*=H$, $\tr_\CC H=0$, and $KJ+JK=0$.

The goal of this subsection is to obtain a sharp algebraic upper bound for 
\[  -\frac{\inner{\sum_\alpha q_{S_\alpha}^{(2)}\omega}{\omega}}
        {2(2n-2)|\omega|^2},\]
with clear equality criteria. See already Proposition \ref{prop:complex-projective-Bochner-algebra}.

\begin{lemma}\label{lem:complex-induced-action}
For $S=H+\tau\Id+K$,
\begin{align*}
        \inner{S^{[2]}\omega}{\omega}
        =
        \frac{2}{2n}\tr_\RR S\,|\omega|^2
        =
        2\tau|\omega|^2 .
\end{align*}
\end{lemma}

\begin{proof}
Choose a unitary basis $z_i$, and put $e_i=z_i$, $f_i=Jz_i$.  Then
$\omega=\sum_i e^i\wedge f^i$ and $|\omega|^2=n$.  The $K$-term contributes
zero, since $K$ has eigenvalues $\mu_i,-\mu_i$ on $e_i,f_i$.  If
$Hz_i=\lambda_i z_i$, then $H$ has eigenvalue $\lambda_i$ on both $e_i,f_i$,
and $\sum_i\lambda_i=0$.  Hence
\begin{align*}
        \inner{S^{[2]}\omega}{\omega}
        &=
        \sum_i
        \inner{S^{[2]}(e^i\wedge f^i)}{e^i\wedge f^i}                 \\
        &=
        \sum_i2(\lambda_i+\tau)
        =
        2n\tau
        =
        2\tau|\omega|^2
        =
        \frac{2}{2n}\tr_\RR S\,|\omega|^2 .
\end{align*}
\end{proof}

\begin{lemma}\label{lem:complex-scalar-identity}
For $S=H+\tau\Id+K$,
\begin{align*}
        -\frac{\inner{q_S^{(2)}\omega}{\omega}}
        {2(2n-2)|\omega|^2}
        =
        \frac{\tr_\CC(H^2)}{n(n-1)}-\tau^2 .
\end{align*}
\end{lemma}

\begin{proof}
Choose the same unitary basis as above.  The $K$-term gives zero, since the
eigenvalue sum on $e^i\wedge f^i$ is $\mu_i-\mu_i=0$.  Mixed terms with $K$
vanish by the $S^1$-symmetry, and the mixed term between $H$ and $\tau\Id$
vanishes because $\tr_\CC H=0$.

For $S=\tau\Id$, the eigenvalue sum on $e^i\wedge f^i$ is $2\tau$ and
$\tr_\RR S=2n\tau$, hence
\begin{align*}
        -\frac{\inner{q_{\tau\Id}^{(2)}\omega}{\omega}}
        {2(2n-2)|\omega|^2}
        =
        -\tau^2 .
\end{align*}
For $H$, diagonalize $Hz_i=\lambda_i z_i$, with $\sum_i\lambda_i=0$.  Since
$|\omega|^2=n$,
\begin{align*}
        -\inner{q_H^{(2)}\omega}{\omega}
        &=
        -\sum_i
        \inner{q_H^{(2)}(e^i\wedge f^i)}{e^i\wedge f^i}               \\
        &=
        -\sum_i(2\lambda_i)(\tr_\RR H-2\lambda_i)
        =
        4\tr_\CC(H^2),                                                \\
        -\frac{\inner{q_H^{(2)}\omega}{\omega}}
        {2(2n-2)|\omega|^2}
        &=
        \frac{4\tr_\CC(H^2)}{4n(n-1)}
        =
        \frac{\tr_\CC(H^2)}{n(n-1)} .
\end{align*}
Adding the two nonzero contributions proves the claim.
\end{proof}

\begin{lemma}\label{lem:complex-projection-bound}
Let $S_\alpha\in\Sym_\RR(\CC^n)$, and write
$S_\alpha=H_\alpha+\tau_\alpha\Id+K_\alpha$.  Then
\begin{align}
        \sum_\alpha\tr_\CC(H_\alpha^2)
        +
        n(n+1)\sum_\alpha\tau_\alpha^2
        \le
        n(n+1)
        \sup_{|z|=1}
        \sum_\alpha\inner{S_\alpha z}{z}^2 .
\label{eq:complex-projection-bound}
\end{align}
If equality holds, then  $K_\alpha = 0$ and   $\sum_\alpha \langle S_\alpha z,z\rangle^2 = \sup_{|z|=1}\sum_\alpha\inner{S_\alpha z}{z}^2$ for every $|z|=1$.
\end{lemma}

\begin{proof}
Set $C=\sup_{|z|=1}\sum_\alpha\inner{S_\alpha z}{z}^2$. Fix any $z\in \CC^n$ with $|z|=1$. Let $z_\theta = \cos\theta z + \sin\theta Jz$.  Since $K_\alpha J = -J K_\alpha$, we have $\langle K_\alpha z_\theta, z_\theta\rangle = \cos(2\theta)\langle K_\alpha z, z\rangle -\sin(2\theta)\langle J K_\alpha z, z\rangle$. Integrating gives
\[
	\int_0^{2\pi}\langle K_\alpha z_\theta, z_\theta\rangle = 0.
\]
Note that  $H_\alpha+\tau_\alpha\Id$ commutes with $J$, thus $\inner{(H_\alpha+\tau_\alpha\Id)z_\theta}{z_\theta} = \inner{(H_\alpha+\tau_\alpha\Id)z}{z}$ for every $\theta$. This implies
\[
	\frac{1}{2\pi}\int_0^{2\pi}\sum_\alpha\inner{S_\alpha z_\theta}{z_\theta}^2 = \sum_\alpha\inner{(H_\alpha+\tau_\alpha\Id)z}{z}^2 + \frac{1}{2\pi}\int_0^{2\pi}\sum_\alpha\inner{K_\alpha z_\theta}{z_\theta}^2.
\]
In particular, the assumption gives
\begin{equation}\label{eq:complex-projection-bound-z}
	\sum_\alpha\inner{(H_\alpha+\tau_\alpha\Id)z}{z}^2 \leq C.
\end{equation}

On the other hand $B_\alpha=H_\alpha+\tau_\alpha\Id$ is Hermitian.
Let $d\sigma$ denote the normalized spherical measure on
$S^{2n-1}\subset\CC^n$, so that
$\int_{S^{2n-1}}d\sigma=1$.  By a unitary change of coordinates, we may
diagonalize $B_\alpha$, say
\[
        \inner{B_\alpha z}{z}
        =
        \sum_{i=1}^n \lambda_i |z_i|^2 .
\]
Since $d\sigma$ is invariant under unitary transformations, Rudin's monomial integral formula \cite[Proposition 1.4.9]{RudinBall} gives
\[
        \int_{S^{2n-1}} |z_i|^4\,d\sigma
        =
        \frac{2}{n(n+1)},
        \qquad
        \int_{S^{2n-1}} |z_i|^2|z_j|^2\,d\sigma
        =
        \frac{1}{n(n+1)}
        \quad (i\neq j).
\]
Therefore
\begin{equation}\label{eq:complex-projection-bound-moment}
\begin{aligned}
        \int_{S^{2n-1}}\inner{B_\alpha z}{z}^2\,d\sigma
        &=
        \sum_i\lambda_i^2
        \int_{S^{2n-1}} |z_i|^4\,d\sigma
        +
        2\sum_{i<j}\lambda_i\lambda_j
        \int_{S^{2n-1}} |z_i|^2|z_j|^2\,d\sigma \\
        &=
        \frac{2}{n(n+1)}\sum_i \lambda_i^2
        +
        \frac{2}{n(n+1)}\sum_{i<j}\lambda_i\lambda_j \\
        &=
        \frac{\tr_\CC(B_\alpha^2)+(\tr_\CC B_\alpha)^2}{n(n+1)} \\
        &=
        \frac{\tr_\CC(H_\alpha^2)+n(n+1)\tau_\alpha^2}{n(n+1)}.
\end{aligned}
\end{equation}
Hence \eqref{eq:complex-projection-bound} follows from combining \eqref{eq:complex-projection-bound-z} and \eqref{eq:complex-projection-bound-moment}. In particular, equality holds in \eqref{eq:complex-projection-bound} if and only if $\sum_\alpha\inner{K_\alpha z_\theta}{z_\theta}^2=0$ and $\sum_\alpha\inner{S_\alpha z_\theta}{z_\theta}^2  = C$ for all $\theta$. This implies the remaining assertion.
\end{proof}

\begin{proposition}[Complex projective Bochner algebra]
\label{prop:complex-projective-Bochner-algebra}
Let $S_\alpha\in\Sym_\RR(\CC^n)$, and write
\begin{align*}
        S_\alpha=H_\alpha+\tau_\alpha\Id+K_\alpha,
        \qquad
        \tau_\alpha=\frac{1}{2n}\tr_\RR S_\alpha .
\end{align*}
Then
\begin{align}
        -\frac{\inner{\sum_\alpha q_{S_\alpha}^{(2)}\omega}{\omega}}
        {2(2n-2)|\omega|^2}
        \le
        \frac{n+1}{n-1}
        \sup_{|z|=1}
        \sum_\alpha\inner{S_\alpha z}{z}^2
        -
        \frac{2n}{n-1}\sum_\alpha\tau_\alpha^2 .
\label{eq:complex-projective-Bochner-algebra}
\end{align}
If equality holds, then equality holds in Lemma
\ref{lem:complex-projection-bound}. 
\end{proposition}

\begin{proof}
By Lemma \ref{lem:complex-scalar-identity},
\begin{align*}
        -\frac{\inner{\sum_\alpha q_{S_\alpha}^{(2)}\omega}{\omega}}
        {2(2n-2)|\omega|^2}
        =
        \frac{\sum_\alpha\tr_\CC(H_\alpha^2)}{n(n-1)}
        -
        \sum_\alpha\tau_\alpha^2 .
\end{align*}
By Lemma \ref{lem:complex-projection-bound},
\begin{align*}
        \sum_\alpha\tr_\CC(H_\alpha^2)
        +
        n(n+1)\sum_\alpha\tau_\alpha^2
        \le
        n(n+1)
        \sup_{|z|=1}
        \sum_\alpha\inner{S_\alpha z}{z}^2 .
\end{align*}
Combining the two estimates gives
\begin{align*}
        -\frac{\inner{\sum_\alpha q_{S_\alpha}^{(2)}\omega}{\omega}}
        {2(2n-2)|\omega|^2}
        &\le
        \frac{n+1}{n-1}
        \sup_{|z|=1}
        \sum_\alpha\inner{S_\alpha z}{z}^2
        -
        \frac{n+1}{n-1}\sum_\alpha\tau_\alpha^2
        -
        \sum_\alpha\tau_\alpha^2                                      \\
        &=
        \frac{n+1}{n-1}
        \sup_{|z|=1}
        \sum_\alpha\inner{S_\alpha z}{z}^2
        -
        \frac{2n}{n-1}\sum_\alpha\tau_\alpha^2 .
\end{align*}
If equality holds in \eqref{eq:complex-projective-Bochner-algebra}, then
equality holds in Lemma \ref{lem:complex-projection-bound}.
\end{proof}

%------------------------------------------------------------
\subsection{The quaternionic projective case}
\label{subsec:quaternionic-projective-algebra}
%------------------------------------------------------------

Let $V=\HH^n$, $n\geq 2$, viewed as a real Euclidean vector space of dimension $4n$. We use $\inner{X}{Y}=\Re\sum_{\alpha=1}^n X_\alpha\overline{Y_\alpha}$, where
$\overline{a+b i+c j+d k}=a-b i-c j-d k$.  Fix an admissible quaternionic frame $I,J,K$, and set
\[
        \omega_\phi(X,Y)=\inner{\phi X}{Y}\quad(\phi\in\{I,J,K\}),
        \qquad
        \Theta=\frac16(\omega_I^2+\omega_J^2+\omega_K^2).
\]
Here $\omega_I^2=\omega_I\wedge\omega_I$, and similarly for $J,K$.  We use the
orthogonal decomposition
\[
        \Sym_\RR(V)
        =
        \Herm_0(\HH^n)\oplus \RR\Id
        \oplus \mathcal S_I\oplus \mathcal S_J\oplus \mathcal S_K,
\]
where
\begin{align*}
        \mathcal S_I
        &=\{T\in\Sym_\RR(V): TI=IT,\ TJ=-JT,\ TK=-KT\},\\
        \mathcal S_J
        &=\{T\in\Sym_\RR(V): TI=-IT,\ TJ=JT,\ TK=-KT\},\\
        \mathcal S_K
        &=\{T\in\Sym_\RR(V): TI=-IT,\ TJ=-JT,\ TK=KT\}.
\end{align*}
Thus every $S\in\Sym_\RR(V)$ is written uniquely as
\[
        S=H+\tau\Id+S_I+S_J+S_K,
        \qquad
        \tau=\frac{1}{4n}\tr_\RR S,
\]
with $H\in\Herm_0(\HH^n)$ and $S_\phi\in\mathcal S_\phi$ for
$\phi\in\{I,J,K\}$.

The goal of this subsection is to obtain a sharp algebraic upper bound for 
\[ -\frac{\inner{\sum_\alpha q_{S_\alpha}^{(4)}\Theta}{\Theta}}
        {4(4n-4)|\Theta|^2},\]
with clear equality criteria. See already Proposition \ref{prop:quaternionic-projective-Bochner-algebra}.

\begin{lemma}\label{lem:theta-counts}
Choose a quaternionic orthonormal basis $e_1,\ldots,e_n$, and use the
associated real orthonormal basis $e_1,Ie_1,Je_1,Ke_1,\ldots,e_n,Ie_n,Je_n,Ke_n$.
Write $\Theta=\sum_A c_A E^A$ in the induced
basis of $\Lambda^4V^*$.  Then
\[
        |\Theta|^2=\frac{n(2n+1)}{3},
        \qquad
        \sum_{\#(A\cap\{e_1,Ie_1\})=1} c_A^2
        =
        \frac{8(n-1)}{9}.
\]
Moreover, if $m_i(A)=\#(A\cap\{e_i,Ie_i,Je_i,Ke_i\})$, then
\[
        \sum_A c_A^2 m_1(A)^2
        -
        \sum_A c_A^2 m_1(A)m_2(A)
        =
        \frac{16(n+1)}{n(2n+1)}|\Theta|^2 .
\]
\end{lemma}

\begin{proof}
Put
\begin{align*}
        \eta_{I,r}
        &= e_r^*\wedge(Ie_r)^*+(Je_r)^*\wedge(Ke_r)^*,\\
        \eta_{J,r}
        &= e_r^*\wedge(Je_r)^*-(Ie_r)^*\wedge(Ke_r)^*,\\
        \eta_{K,r}
        &= e_r^*\wedge(Ke_r)^*+(Ie_r)^*\wedge(Je_r)^* .
\end{align*}
Then $\omega_\phi=\sum_r\eta_{\phi,r}$ for $\phi\in\{I,J,K\}$, and
\[
        \eta_{I,r}^2=\eta_{J,r}^2=\eta_{K,r}^2
        =2e_r^*\wedge(Ie_r)^*\wedge(Je_r)^*\wedge(Ke_r)^*.
\]
It follows that
\begin{align*}
        \Theta
        &=
        \sum_r e_r^*\wedge(Ie_r)^*\wedge(Je_r)^*\wedge(Ke_r)^*\\
        &\quad
        +\frac13\sum_{r<s}
        \big(
        \eta_{I,r}\wedge\eta_{I,s}
        +\eta_{J,r}\wedge\eta_{J,s}
        +\eta_{K,r}\wedge\eta_{K,s}
        \big).
\end{align*}
The summands displayed above are mutually orthogonal, and
$|\eta_{\phi,r}\wedge\eta_{\phi,s}|^2=4$ for
$\phi\in\{I,J,K\}$.  Hence
\[
        |\Theta|^2=n+\frac19\binom n2\cdot 3\cdot 4
        =\frac{n(2n+1)}3.
\]
For fixed $s\ne1$, the only summands which contain exactly one of
$e_1,Ie_1$ are
\[
        \eta_{J,1}\wedge\eta_{J,s},
        \qquad
        \eta_{K,1}\wedge\eta_{K,s}.
\]
Each has four basis terms, all with coefficient squared $1/9$.  Therefore
\[
        \sum_{\#(A\cap\{e_1,Ie_1\})=1} c_A^2
        =(n-1)\frac{4+4}{9}
        =\frac{8(n-1)}9.
\]

With $m_i(A)$ as above, $\sum_i m_i(A)=4$ for every $A$, and symmetry gives
\[
        \sum_A c_A^2 m_1(A)
        =
        \frac4n|\Theta|^2 .
\]
For each multi-index $A$, the number
\[
        m_1(A)\big(4-m_1(A)\big)
\]
counts the unordered pairs in $\{e_1,Ie_1,Je_1,Ke_1\}$ which meet
$A$ in exactly one element.  There are six such pairs, and each pair
has the same coefficient sum as $\{e_1,Ie_1\}$.  Hence
\begin{align*}
        \sum_A c_A^2 m_1(A)^2
        &=
        4\sum_A c_A^2 m_1(A)
        -
        \sum_A c_A^2 m_1(A)\big(4-m_1(A)\big)                \\
        &=
        \frac{16}{n}|\Theta|^2
        -6\frac{8(n-1)}9 .
\end{align*}
Also, since $\sum_i m_i(A)=4$,
\begin{align*}
        16|\Theta|^2
        &=
        \sum_A c_A^2
        \left(
        \sum_{i=1}^n m_i(A)
        \right)^2                                                        \\
        &=
        n\sum_A c_A^2 m_1(A)^2
        +
        n(n-1)\sum_A c_A^2 m_1(A)m_2(A) .
\end{align*}
Substituting the preceding value of $\sum_A c_A^2 m_1(A)^2$ gives
\[
        \sum_A c_A^2 m_1(A)^2
        -
        \sum_A c_A^2 m_1(A)m_2(A)
        =
        \frac{16}{n}|\Theta|^2-\frac{16n}{3}
        =
        \frac{16(n+1)}{n(2n+1)}|\Theta|^2 .
\]
\end{proof}

\begin{lemma}\label{lem:quaternionic-induced-action}
For $S=H+\tau\Id+S_I+S_J+S_K$,
\[
        \inner{S^{[4]}\Theta}{\Theta}
        =
        \frac4{4n}\tr_\RR S\,|\Theta|^2
        =4\tau|\Theta|^2 .
\]
\end{lemma}

\begin{proof}
The terms $S_I,S_J,S_K$ vanish by $Sp(1)$-symmetry.  For instance,
$J^{-1}S_IJ=-S_I$ and $J^*\Theta=\Theta$, so
\[
        \inner{S_I^{[4]}\Theta}{\Theta}
        =
        \inner{(J^{-1}S_IJ)^{[4]}\Theta}{\Theta}
        =
        -\inner{S_I^{[4]}\Theta}{\Theta}.
\]
The same argument applies to $S_J$ and $S_K$.

Choose a quaternionic orthonormal basis diagonalizing $H$:
$He_i=\lambda_i e_i$, $\sum_i\lambda_i=0$.  Using the notation of
Lemma~\ref{lem:theta-counts},
\[
        \inner{H^{[4]}\Theta}{\Theta}
        =
        \sum_i\lambda_i
        \sum_A c_A^2 m_i(A)
        =
        \frac4n|\Theta|^2\sum_i\lambda_i
        =0.
\]
Finally, $(\tau\Id)^{[4]}$ acts on $\Lambda^4V^*$ by multiplication by
$4\tau$.  This proves the assertion.
\end{proof}

\begin{lemma}\label{lem:quaternionic-scalar-identity}
For $S=H+\tau\Id+S_I+S_J+S_K$,
\[
        -\frac{\inner{q_S^{(4)}\Theta}{\Theta}}
        {4(4n-4)|\Theta|^2}
        =
        \frac{n+1}{n(n-1)(2n+1)}\tr_\HH(H^2)
        -\tau^2
        +\frac{1}{12n(2n+1)}
        \sum_{\phi\in\{I,J,K\}}|S_\phi|^2 .
\]
\end{lemma}

\begin{proof}
Since
\[
        q_S^{(4)}=(\tr_\RR S)S^{[4]}-(S^{[4]})^2,
\]
the map
\[
        S\longmapsto -\inner{q_S^{(4)}\Theta}{\Theta}
\]
is a quadratic form on $\Sym_\RR(V)$.  By invariance under the
$Sp(n)Sp(1)$-symmetry preserving $\Theta$, the polarized form has no mixed
terms among
\[
        \Herm_0(\HH^n),
        \qquad
        \mathcal S_I,
        \qquad
        \mathcal S_J,
        \qquad
        \mathcal S_K.
\]
The mixed term between $\tau\Id$ and a trace-free component $R$ is proportional
to $\inner{R^{[4]}\Theta}{\Theta}$, which is zero by
Lemma~\ref{lem:quaternionic-induced-action}.  Therefore
\begin{align*}
        -\inner{q_S^{(4)}\Theta}{\Theta}
        &=
        -\inner{q_H^{(4)}\Theta}{\Theta}
        -\inner{q_{\tau\Id}^{(4)}\Theta}{\Theta}                         \\
        &\quad
        -\inner{q_{S_I}^{(4)}\Theta}{\Theta}
        -\inner{q_{S_J}^{(4)}\Theta}{\Theta}
        -\inner{q_{S_K}^{(4)}\Theta}{\Theta} .
\end{align*}

For $S=\tau\Id$, one has
\[
        q_{\tau\Id}^{(4)}\Theta=16(n-1)\tau^2\Theta,
\]
and hence
\[
        -\frac{\inner{q_{\tau\Id}^{(4)}\Theta}{\Theta}}
        {4(4n-4)|\Theta|^2}
        =-\tau^2 .
\]

Next suppose $H\in\Herm_0(\HH^n)$.  Choose a quaternionic orthonormal basis
with $He_i=\lambda_i e_i$ and $\sum_i\lambda_i=0$.  Then
$\tr_\HH(H^2)=\sum_i\lambda_i^2$, and, with $m_i(A)$ as in
Lemma~\ref{lem:theta-counts},
\[
        H^{[4]}E^A
        =
        \left(\sum_i\lambda_i m_i(A)\right)E^A.
\]
Since $\tr_\RR H=0$, symmetry and $\sum_i\lambda_i=0$ give
\begin{align*}
        -\inner{q_H^{(4)}\Theta}{\Theta}
        &=
        \sum_A c_A^2
        \left(\sum_i\lambda_i m_i(A)\right)^2                         \\
        &=
        \sum_A c_A^2 m_1(A)^2\sum_i\lambda_i^2
        +2\sum_A c_A^2 m_1(A)m_2(A)\sum_{i<j}\lambda_i\lambda_j       \\
        &=
        \left(
        \sum_A c_A^2 m_1(A)^2
        -
        \sum_A c_A^2 m_1(A)m_2(A)
        \right)
        \sum_i\lambda_i^2                                             \\
        &=
        \frac{16(n+1)}{n(2n+1)}|\Theta|^2\tr_\HH(H^2).
\end{align*}
Thus
\[
        -\frac{\inner{q_H^{(4)}\Theta}{\Theta}}
        {4(4n-4)|\Theta|^2}
        =
        \frac{n+1}{n(n-1)(2n+1)}\tr_\HH(H^2).
\]

It remains to compute one of the three equivalent components.  Let
$T\in\mathcal S_I$.  We may choose a quaternionic orthonormal basis such that
\[
        Te_i=\mu_i e_i,
        \qquad
        T(Ie_i)=\mu_i Ie_i,
        \qquad
        T(Je_i)=-\mu_i Je_i,
        \qquad
        T(Ke_i)=-\mu_i Ke_i.
\]
Then $\tr_\RR T=0$ and $|T|^2=4\sum_i\mu_i^2$.  In the expansion of $\Theta$,
only the terms $\eta_{I,r}\wedge\eta_{I,s}$ contribute.  For fixed $r<s$, the
four eigenvalue sums are
\[
        2\mu_r+2\mu_s,
        \qquad
        2\mu_r-2\mu_s,
        \qquad
        -2\mu_r+2\mu_s,
        \qquad
        -2\mu_r-2\mu_s.
\]
The sum of their squares is $16(\mu_r^2+\mu_s^2)$.  Since all four coefficients
have square $1/9$,
\[
        -\inner{q_T^{(4)}\Theta}{\Theta}
        =
        \frac19\sum_{r<s}16(\mu_r^2+\mu_s^2)
        =
        \frac{4(n-1)}9|T|^2 .
\]
Using $|\Theta|^2=n(2n+1)/3$, we get
\[
        -\frac{\inner{q_T^{(4)}\Theta}{\Theta}}
        {4(4n-4)|\Theta|^2}
        =
        \frac{|T|^2}{12n(2n+1)}.
\]
The same computation applies to $\mathcal S_J$ and $\mathcal S_K$.  Combining
the five terms gives the stated identity.
\end{proof}

\begin{lemma}\label{lem:quaternionic-projection-bound}
Let $S_\alpha\in\Sym_\RR(\HH^n)$, and write $S_\alpha=H_\alpha+\tau_\alpha\Id+(S_\alpha)_I+(S_\alpha)_J+(S_\alpha)_K$. Then
\begin{equation}\label{eq:quaternionic-projection-bound}
\begin{aligned}
        &\sum_\alpha\tr_\HH(H_\alpha^2)
        +\frac14\sum_\alpha\sum_{\phi\in\{I,J,K\}}|(S_\alpha)_\phi|^2
        +n(2n+1)\sum_\alpha\tau_\alpha^2                         \\
        &\le
        n(2n+1)
        \sup_{|q|=1}\sum_\alpha\inner{S_\alpha q}{q}^2 .
\end{aligned}
\end{equation}
If equality holds, then
\[
        \sum_\alpha\inner{S_\alpha q}{q}^2
        =
        \sup_{|u|=1}\sum_\alpha\inner{S_\alpha u}{u}^2
\]
for every $|q|=1$.
\end{lemma}

\begin{proof}
Let $d\sigma$ be the normalized measure on $S^{4n-1}\subset V$ and $d\omega$ be the Euclidean surface measure on $S^{4n-1}$.  By polar
coordinates,
\begin{align*}
        \left(\int_0^\infty r^{4n+3}e^{-r^2/2}\,dr\right)
        \left(\int_{S^{4n-1}}q_iq_jq_kq_l\,d\omega\right) 
        & =
        \int_{\RR^{4n}}x_ix_jx_kx_l e^{-|x|^2/2}\,dx \\
        & =
        (2\pi)^{2n}
        \begin{cases}
        3, & i=j=k=l,\\
        1, & i=j,\ k=l,\ i\ne k,\\
        1, & i=k,\ j=l,\ i\ne j,\\
        1, & i=l,\ j=k,\ i\ne j,\\
        0, & \text{otherwise}
        \end{cases} \\
        & =
        (2\pi)^{2n}
        \left(
        \delta_{ij}\delta_{kl}
        +
        \delta_{ik}\delta_{jl}
        +
        \delta_{il}\delta_{jk}
        \right).
\end{align*}
Here the third equality follows from the independence of the Gaussian
coordinates and the one-dimensional moments
$\frac{1}{\sqrt{2\pi}}\int_\RR t^2e^{-t^2/2}\,dt=1$ and
$\frac{1}{\sqrt{2\pi}}\int_\RR t^4e^{-t^2/2}\,dt=3$.  Moreover,
\begin{align*}
        |S^{4n-1}|
        \int_0^\infty r^{4n-1}e^{-r^2/2}\,dr
        &=
        (2\pi)^{2n},
        &
        \frac{
        \int_0^\infty r^{4n+3}e^{-r^2/2}\,dr
        }{
        \int_0^\infty r^{4n-1}e^{-r^2/2}\,dr
        }
        &=
        4n(4n+2).
\end{align*}
Since $d\sigma=|S^{4n-1}|^{-1}d\omega$, we obtain
\[
        \int_{S^{4n-1}} q_iq_jq_kq_l\,d\sigma
        =
        \frac{
        \delta_{ij}\delta_{kl}
        +
        \delta_{ik}\delta_{jl}
        +
        \delta_{il}\delta_{jk}}
        {4n(4n+2)}.
\]
Consequently,  for $R,T\in\Sym_\RR(V)$,
\begin{equation}\label{eq:spherical-fourth-moment-quaternionic}
        \int_{S^{4n-1}}
        \inner{Rq}{q}\inner{Tq}{q}\,d\sigma
        =
        \frac{\tr_\RR R\,\tr_\RR T+2\tr_\RR(RT)}{4n(4n+2)} .
\end{equation}
The decomposition of $\Sym_\RR(V)$ used above is orthogonal, and all summands
except $\RR\Id$ are trace-free.  Hence the mixed terms vanish after
integrating.  Thus
\[
        \int_{S^{4n-1}}
        \inner{(H_\alpha+\tau_\alpha\Id)q}{q}^2\,d\sigma
        =
        \frac{\tr_\HH(H_\alpha^2)+n(2n+1)\tau_\alpha^2}{n(2n+1)}.
\]
For each $\phi\in\{I,J,K\}$,
\[
        \int_{S^{4n-1}}
        \inner{(S_\alpha)_\phi q}{q}^2\,d\sigma
        =
        \frac{|(S_\alpha)_\phi|^2}{4n(2n+1)}.
\]
Therefore
\begin{align*}
        &\sup_{|q|=1}\sum_\alpha\inner{S_\alpha q}{q}^2   \\              
        &\ge
        \int_{S^{4n-1}}
        \sum_\alpha\inner{S_\alpha q}{q}^2\,d\sigma                     \\
        &=
        \frac{1}{n(2n+1)}
        \sum_\alpha
        \left(
        \tr_\HH(H_\alpha^2)+n(2n+1)\tau_\alpha^2
        \right)                                                  +
        \frac{1}{4n(2n+1)}
        \sum_\alpha\sum_{\phi\in\{I,J,K\}}|(S_\alpha)_\phi|^2 .
\end{align*}
Multiplying by $n(2n+1)$ gives \eqref{eq:quaternionic-projection-bound}.

If equality holds in \eqref{eq:quaternionic-projection-bound}, then the
continuous function
\[
        q\longmapsto \sum_\alpha\inner{S_\alpha q}{q}^2
\]
has average equal to its supremum on the unit sphere.  Hence it is identically
equal to its supremum.
\end{proof}

\begin{proposition}[Quaternionic projective Bochner algebra]
\label{prop:quaternionic-projective-Bochner-algebra}
Let $S_\alpha\in\Sym_\RR(\HH^n)$, and write
\[
        S_\alpha
        =
        H_\alpha+\tau_\alpha\Id
        +(S_\alpha)_I+(S_\alpha)_J+(S_\alpha)_K,
        \qquad
        \tau_\alpha=\frac{1}{4n}\tr_\RR S_\alpha .
\]
Then
\begin{align}
        -\frac{\inner{\sum_\alpha q_{S_\alpha}^{(4)}\Theta}{\Theta}}
        {4(4n-4)|\Theta|^2}
        \le
        \frac{n+1}{n-1}
        \sup_{|q|=1}
        \sum_\alpha\inner{S_\alpha q}{q}^2
        -
        \frac{2n}{n-1}\sum_\alpha\tau_\alpha^2 .
\label{eq:quaternionic-projective-Bochner-algebra}
\end{align}
If equality holds, then equality holds in
Lemma~\ref{lem:quaternionic-projection-bound}, and
\[
        (S_\alpha)_I=(S_\alpha)_J=(S_\alpha)_K=0
        \qquad
        \text{for every }\alpha .
\]
\end{proposition}

\begin{proof}
By Lemma~\ref{lem:quaternionic-scalar-identity},
\begin{align*}
        &-\frac{\inner{\sum_\alpha q_{S_\alpha}^{(4)}\Theta}{\Theta}}
        {4(4n-4)|\Theta|^2}                                             \\
        &\quad=
        \frac{n+1}{n(n-1)(2n+1)}
        \left(
        \sum_\alpha\tr_\HH(H_\alpha^2)
        +\frac14\sum_\alpha\sum_{\phi\in\{I,J,K\}}|(S_\alpha)_\phi|^2
        +n(2n+1)\sum_\alpha\tau_\alpha^2
        \right)                                                          \\
        &\qquad-
        \frac{2n}{n-1}\sum_\alpha\tau_\alpha^2
        -
        \frac{n+2}{6n(n-1)(2n+1)}
        \sum_\alpha\sum_{\phi\in\{I,J,K\}}|(S_\alpha)_\phi|^2 .
\end{align*}
Lemma~\ref{lem:quaternionic-projection-bound} gives
\begin{align*}
        &\sum_\alpha\tr_\HH(H_\alpha^2)
        +\frac14\sum_\alpha\sum_{\phi\in\{I,J,K\}}|(S_\alpha)_\phi|^2
        +n(2n+1)\sum_\alpha\tau_\alpha^2                         \\
        &\le
        n(2n+1)
        \sup_{|q|=1}\sum_\alpha\inner{S_\alpha q}{q}^2 .
\end{align*}
Hence
\begin{align*}
        &-\frac{\inner{\sum_\alpha q_{S_\alpha}^{(4)}\Theta}{\Theta}}
        {4(4n-4)|\Theta|^2}                                             \\
        &\quad\le
        \frac{n+1}{n-1}
        \sup_{|q|=1}\sum_\alpha\inner{S_\alpha q}{q}^2
        -
        \frac{2n}{n-1}\sum_\alpha\tau_\alpha^2                          \\
        &\qquad-
        \frac{n+2}{6n(n-1)(2n+1)}
        \sum_\alpha\sum_{\phi\in\{I,J,K\}}|(S_\alpha)_\phi|^2             \\
        &\quad\le
        \frac{n+1}{n-1}
        \sup_{|q|=1}\sum_\alpha\inner{S_\alpha q}{q}^2
        -
        \frac{2n}{n-1}\sum_\alpha\tau_\alpha^2 .
\end{align*}
This proves \eqref{eq:quaternionic-projective-Bochner-algebra}.

If equality holds in \eqref{eq:quaternionic-projective-Bochner-algebra}, then
\begin{align*}
        0
        &=
        \frac{n+1}{n(n-1)(2n+1)}
        \left(
        n(2n+1)
        \sup_{|q|=1}\sum_\alpha\inner{S_\alpha q}{q}^2                  
        \right.                                                           \\
        &\qquad\left.
        -\sum_\alpha\tr_\HH(H_\alpha^2)
        -\frac14\sum_\alpha\sum_{\phi\in\{I,J,K\}}|(S_\alpha)_\phi|^2
        -n(2n+1)\sum_\alpha\tau_\alpha^2
        \right)                                                          \\
        &\quad+
        \frac{n+2}{6n(n-1)(2n+1)}
        \sum_\alpha\sum_{\phi\in\{I,J,K\}}|(S_\alpha)_\phi|^2 .
\end{align*}
Both terms on the right are nonnegative.  Hence equality holds in
Lemma~\ref{lem:quaternionic-projection-bound}, and
\[
        \sum_\alpha\sum_{\phi\in\{I,J,K\}}|(S_\alpha)_\phi|^2=0.
\]
Therefore
\[
        (S_\alpha)_I=(S_\alpha)_J=(S_\alpha)_K=0
        \qquad
        \text{for every }\alpha .
\]
The final assertion of Lemma~\ref{lem:quaternionic-projection-bound} then gives
\[
        \sum_\alpha\inner{S_\alpha q}{q}^2
        =
        \sup_{|u|=1}\sum_\alpha\inner{S_\alpha u}{u}^2
\]
for every unit $q\in\HH^n$.
\end{proof}

%============================================================
\section{Proof of the estimate in Theorem \ref{thm:main-ball-FPn}}
\label{sec:proof-estimate-main}
%============================================================

We treat the two cases simultaneously. Set $(m,p,\Psi,\mathbb F)=(2n,2,\omega,\CC)$ in
the almost Hermitian case and $(m,p,\Psi,\mathbb F)=(4n,4,\Theta,\HH)$ in the
almost quaternion-Hermitian case.

Since $\Psi$ is harmonic, the Bochner formula and the Euclidean Gauss equation
give
\begin{align*}
        0
        &=
        \int_\Sigma |\nabla\Psi|^2\,d\mu
        +
        \int_\Sigma
        \left\langle
        \sum_\alpha q_{A_\alpha}^{(p)}\Psi,\Psi
        \right\rangle\,d\mu .
\end{align*}
Equivalently,
\begin{align}
        \int_\Sigma
        \left(
        -\frac{
        \left\langle
        \sum_\alpha q_{A_\alpha}^{(p)}\Psi,\Psi
        \right\rangle}
        {p(m-p)|\Psi|^2}
        \right)\,d\mu
        =
        \int_\Sigma
        \frac{|\nabla\Psi|^2}{p(m-p)|\Psi|^2}\,d\mu
        \ge 0 .
\label{eq:estimate-harmonic-Bochner}
\end{align}

We apply the algebraic estimate pointwise with $S_\alpha=A_\alpha$.  Since
$\tau_\alpha=m^{-1}\tr A_\alpha$, Proposition
\ref{prop:complex-projective-Bochner-algebra} in the complex case and
Proposition \ref{prop:quaternionic-projective-Bochner-algebra} in the
quaternionic case give
\begin{align}
        -\frac{
        \left\langle
        \sum_\alpha q_{A_\alpha}^{(p)}\Psi,\Psi
        \right\rangle}
        {p(m-p)|\Psi|^2}
        \le
        \frac{n+1}{n-1}\sup_{|v|=1}|A(v,v)|^2
        -
        \frac{2n}{n-1}\frac{1}{m^2}|\tr A|^2 .
\label{eq:estimate-pointwise-Bochner-algebra}
\end{align}
Combining \eqref{eq:estimate-harmonic-Bochner} and
\eqref{eq:estimate-pointwise-Bochner-algebra}, we obtain
\begin{align}
        \int_\Sigma
        \sup_{|v|=1}|A(v,v)|^2\,d\mu
        \ge
        \frac{2n}{n+1}
        \int_\Sigma
        \frac{1}{m^2}|\tr A|^2\,d\mu .
\label{eq:estimate-supA-trace-bound}
\end{align}

It remains to estimate the trace term.  Since
$\operatorname{div}_\Sigma F^T=m+\tr A_{F^\perp}$ and $\Sigma$ is closed, we have the Minkowski identity
\[
        \int_\Sigma \tr A_{F^\perp}\,d\mu=-m\Vol(\Sigma).
\]
Therefore, by Cauchy's inequality and $|F^\perp|\le |F|\le1$,
\begin{equation}\label{eq:estimate-trace-volume-bound}
\begin{aligned}
        \Vol(\Sigma)^2
        &=
        \left(
        \int_\Sigma
        \frac{1}{m}\tr A_{F^\perp}\,d\mu
        \right)^2                                                     \\
        &\le
        \left(
        \int_\Sigma
        \frac{1}{m^2}|\tr A|^2\,d\mu
        \right)
        \left(
        \int_\Sigma |F^\perp|^2\,d\mu
        \right)                                                       \\
        &\le
        \Vol(\Sigma)
        \int_\Sigma
        \frac{1}{m^2}|\tr A|^2\,d\mu .
\end{aligned}
\end{equation}
Hence
\begin{align}
        \int_\Sigma
        \frac{1}{m^2}|\tr A|^2\,d\mu
        \ge
        \Vol(\Sigma).
\label{eq:estimate-trace-volume-lower-bound}
\end{align}
By the definition of $\kappa(F)$,
\eqref{eq:estimate-supA-trace-bound}, and
\eqref{eq:estimate-trace-volume-lower-bound},
\begin{align*}
        \kappa(F)^2\Vol(\Sigma)
        &\ge
        \int_\Sigma
        \sup_{|v|=1}|A(v,v)|^2\,d\mu                                  \\
        &\ge
        \frac{2n}{n+1}
        \int_\Sigma
        \frac{1}{m^2}|\tr A|^2\,d\mu                                  \\
        &\ge
        \frac{2n}{n+1}\Vol(\Sigma).
\end{align*}
Thus
\[
        \kappa(F)^2\ge\frac{2n}{n+1}.
\]
This proves the estimate in Theorem \ref{thm:main-ball-FPn}.
%============================================================
\section{Proof of the rigidity in Theorem \ref{thm:main-ball-FPn}}
\label{sec:rigidity}
%============================================================

We prove the equality case.  Set $(m,p,\Psi,\mathbb F)=(2n,2,\omega,\CC)$ in
the almost Hermitian case and $(m,p,\Psi,\mathbb F)=(4n,4,\Theta,\HH)$ in the
almost quaternion-Hermitian case.  Equality in the estimate gives equality in
\eqref{eq:estimate-harmonic-Bochner}, hence
\[
        \int_\Sigma
        \frac{|\nabla\Psi|^2}{p(m-p)|\Psi|^2}\,d\mu=0 .
\]
Since $|\omega|^2=n$ and $|\Theta|^2=n(2n+1)/3$ pointwise, by
Lemmas~\ref{lem:complex-induced-action} and \ref{lem:theta-counts}, we get
$\nabla\Psi=0$.  Thus $(\Sigma,g,J)$ is K\"ahler in the complex case, and
$(\Sigma,g,\calQ)$ is quaternionic-K\"ahler in the quaternionic case.

Equality in \eqref{eq:estimate-trace-volume-bound} gives
\[
        |F^\perp|=|F|=1
        \qquad\text{on }\Sigma .
\label{eq:rigidity-image-boundary}
\]
Hence $F^T=0$ and $F(\Sigma)\subset S(1)$. Since $F$ is an isometric immersion, we simply identify $T_x\Sigma$ with $dF_x(T_x\Sigma)$.  We write
$N^{\RR}\Sigma=\RR F\oplus N^S\Sigma$, where $N^S\Sigma$ is the spherical
normal bundle of $F:\Sigma\looparrowright S(1)$.  Equality in Cauchy's
inequality also gives
\[
        \tr A=-mF .
\]
If $A^S$ denotes the second fundamental form of
$F:\Sigma\looparrowright S(1)$, then
\[
        A=A^S-gF .
\]
Therefore $\tr A^S=0$.

\begin{lemma}\label{lem:equality-spherical-shape-package}
We have $A^S_\eta\in \Herm_0(T_x\Sigma)$ for every $\eta\in N^S_x\Sigma$. Moreover,
\[
		|A^S(v,v)|^2 = \frac{n-1}{n+1}
\]
for every unit $v\in T_x\Sigma$.
\end{lemma}

\begin{proof}
	Choose an orthonormal basis $\nu_0,\ldots,\nu_{l-1}$ of $N_x\Sigma$ such that $\nu_0 = F$ and $\nu_\alpha = \eta_\alpha$ for $\alpha\geq 1$. Equality in Proposition \ref{prop:complex-projective-Bochner-algebra} in the Kähler case and Proposition \ref{prop:quaternionic-projective-Bochner-algebra} in the quaternionic-Kähler case give $(A_\alpha)_x\in \Herm(T_x\Sigma)$ and 
\[
	\sum_\alpha\langle A_{\alpha} v,v\rangle^2= \frac{2n}{n+1}
\]
for every unit $v\in T_x\Sigma$ and every $x\in\Sigma$. Since $\tr A^S = 0$, this implies $A^S_\eta\in \Herm_0(T_x\Sigma)$ for every $\eta\in N^S_x\Sigma$.  Moreover, $A_0 = -I$ and $A_\alpha = A^S_{\eta_\alpha}$ for $\alpha\geq 1$, hence the above identity implies
\[
	  \sum_{\alpha\geq 1}\langle A^S_{\eta_\alpha} v,v\rangle^2 =   \frac{2n}{n+1} - \langle -\Id v,v\rangle^2 =  \frac{n-1}{n+1},
\]
which implies the assertion.
\end{proof}

\begin{proposition}\label{lem:equality-intrinsic-space-form}
There is an isometry 
\[\varphi:(\Sigma,g)\to (\mathbb FP^n, g_{FS}),\qquad \mathbb F = \CC, \HH,\]
such that $d\varphi\circ J =J_{\CC P^n}\circ d\varphi$ in the K\"ahler case, and $        d\varphi\circ\mathcal Q\circ d\varphi^{-1}=\mathcal Q_{\HH P^n}$
in the quaternionic-K\"ahler case. Here $g_{FS}$ is the
standard projective metric normalized with constant holomorphic/quaternionic sectional curvature equals to $K = \frac{2n}{n+1}$.
	In particular, we have
\begin{equation}\label{eq:complex-curvature}
        R^\Sigma(X,Y)
        =
        \frac{K}{4}
        \bigl(
        X\wedge Y+JX\wedge JY
        +2\langle JX,Y\rangle J
        \bigr),
\end{equation}
in the K\"ahler case, and
\begin{equation}\label{eq:quaternionic-curvature}
        R^\Sigma(X,Y)
        =
        \frac{K}{4}
        \left(
        X\wedge Y
        +\sum_{\phi=I,J,K}
        \bigl(
        \phi X\wedge\phi Y
        +2\langle\phi X,Y\rangle\phi
        \bigr)
        \right)
\end{equation}
in the quaternionic-K\"ahler case. Here we use the convention $(X\wedge Y)Z=\langle X,Z\rangle Y-\langle Y,Z\rangle X$.
\end{proposition}

\begin{proof}
	Let $\phi$ be a compatible unit complex structure. In the K\"ahler case $\phi = J$. In the quaternionic-K\"ahler case $\phi \in\calQ=\Span\{I,J,K\}$. Let $\eta\in N^S_x\Sigma$. Since $A^S_\eta\in \Herm_0(T_x\Sigma)$ by Lemma \ref{lem:equality-spherical-shape-package}, it commutes with $\phi$. Hence, for every unit vector $v\in T_x\Sigma$, 
	\[	\langle A^S(v, \phi v),\eta\rangle = \langle A^S_\eta v, \phi v\rangle = 0\]
	and
	\[
		\langle A^S(\phi v, \phi v),\eta\rangle = \langle A^S_\eta \phi v, \phi v\rangle = \langle A^S_\eta v,  v\rangle = \langle A^S(v, v),\eta\rangle.
	\]
	Since this holds for every $\eta\in N^S_x\Sigma$, we obtain
	\begin{align*}
		A^S(v, \phi v) = 0,
		\qquad
		 A^S(\phi v, \phi v) =  A^S(v, v).
	\end{align*}
    Then the Gauss equation and Lemma \ref{lem:equality-spherical-shape-package} give
    \begin{equation*}
        \begin{aligned}
		R^\Sigma(v, \phi v, v, \phi v) &= 1 + \langle A^S(v, v), A^S(\phi v, \phi v)\rangle - |A^S(v, \phi v)|^2\\
		&= 1 + |A^S(v,v)|^2\\
		&= \frac{2n}{n+1},            
        \end{aligned}
    \end{equation*}
	for every $v\in T_x\Sigma$ and every such $\phi$.

    Therefore $(\Sigma, g)$ has constant positive holomorphic sectional curvature in the K\"ahler case and constant positive quaternionic sectional curvature in the quaternionic-K\"ahler case. Standard classification results imply that $(\Sigma,g)$ is locally isometric to the positive complex space form in the K\"ahler case, respectively to the positive quaternionic space form in the quaternionic-K\"ahler case, with constant holomorphic, respectively quaternionic, sectional curvature $K = \frac{2n}{n+1}$; see \cite{Nomizu1973} for the K\"ahler case and \cite{Ishihara1974, Alekseevskii1968} for the quaternionic-K\"ahler case. This proves \eqref{eq:complex-curvature} and \eqref{eq:quaternionic-curvature}. In particular, all sectional curvatures are positive.

     Moreover, $\Sigma$ is orientable: $\omega^n\neq 0$ in the K\"ahler case, and $\Theta^n\neq 0$ in the quaternionic-K\"ahler case. Since $\Sigma$ is closed, connected, even-dimensional, orientable, and has positive sectional curvature, Synge's theorem implies that $\Sigma$ is simply connected. The local model is therefore the simply connected positive complex or quaternionic space form, and the local isometry extends to a global isometry $\varphi:\Sigma\to\mathbb FP^n$. Choosing the initial local isometry to preserve $J$ in the K\"ahler case, respectively $\calQ$ in the quaternionic-K\"ahler case, gives the stated compatibility because these structures are parallel.

% Choose points $x_0\in\Sigma$, $o\in \mathbb FP^n$, and a linear isometry
% \[
%         I_0:
%         T_{x_0}\Sigma
%         \longrightarrow
%         T_o \mathbb FP^n
% \]
% which is complex linear in the K\"ahler case and, and satisfies $I_0\mathcal Q_{\Sigma,x_0}I_0^{-1}=\mathcal Q_{\HH P^n,o}$ in  the quaternionic-K\"ahler case.  The curvature formulas
% give
% \[
%         I_0\bigl(R^{\Sigma}(X,Y)Z\bigr)
%         =
%         R^{\mathbb FP^n}(I_0X,I_0Y)I_0Z.
% \]
% Since both curvature tensors are parallel, this identity is preserved under
% parallel translation along corresponding curves.  Since $\Sigma$ and $\mathbb FP^n$ are both complete and simply connected, 
% \cite[Main Theorem]{Ambrose1956} therefore gives a global isometry
% \[
%         \varphi:
%         \Sigma
%         \longrightarrow
%         \mathbb FP^n
% \]
% satisfying $d\varphi_{x_0}=I_0$. Since the complex structure, respectively the quaternionic structure bundle, is parallel, $d\varphi$ preserves this structure everywhere.

% Since $J$ and \(J_{\CC P^n}\) are parallel in the K\"ahler case,
% and since \(\mathcal Q\) and \(\mathcal Q_{\HH P^n}\) are parallel
% in the quaternionic-K\"ahler case, this initial compatibility is preserved
% by parallel translation.  Hence
% \[
%         d\varphi_x\circ J_{\Sigma,x}
%         =
%         J_{\mathbb F,\varphi(x)}\circ d\varphi_x
% \]
% in the K\"ahler case, and
% \[
%         d\varphi_x\mathcal Q_{\Sigma,x}d\varphi_x^{-1}
%         =
%         \mathcal Q_{\mathbb F,\varphi(x)}
% \]
% in the quaternionic-K\"ahler case.

 \end{proof}

%------------------------------------------------------------
\subsection{Parallel normal bundle in the sphere}
\label{subsec:rigidity-parallel-normal-bundle}
%------------------------------------------------------------

We use the following elementary fact.

\begin{lemma}\label{lem:rigidity-first-prolongation}
Let $T:T_x\Sigma\to\End(T_x\Sigma)$ be real linear.  Assume that each $T_X$ is
trace-free Hermitian over $\mathbb F$ and $T_XY=T_YX$ for all
$X,Y\in T_x\Sigma$.  Then $T=0$.
\end{lemma}

\begin{proof}
Choose $\phi=J$ in the K\"ahler case and a local unit
$\phi\in\calQ$ in the quaternionic-K\"ahler case.  Since $T_X$ is
Hermitian,
\begin{align*}
        T_{\phi X}Y
        =
        T_Y(\phi X)
        =
        \phi T_YX
        =
        \phi T_XY .
\end{align*}
Thus $T_{\phi X}=\phi T_X$.  The left-hand side is self-adjoint,
while the right-hand side is skew-adjoint.  Hence $T_{\phi X}=0$ for all
$X$, and $T=0$.
\end{proof}

We next consider the first normal bundle
\begin{align*}
        E_x=\Span\{A^S_x(X,Y):X,Y\in T_x\Sigma\}
        \subset N_x^S\Sigma .
\end{align*}

\begin{lemma}\label{lem:parallel-E}
	$E$ is a parallel subbundle of $N^S\Sigma$. Moreover, $ \nabla^{\perp,S}A^S=0$, so that $A^S$ is parallel.
\end{lemma}

\begin{proof}
	We first show that $E$ is parallel in $N^S\Sigma$.  Let $\zeta$ be a local
normal field with $\zeta\perp E$.  Then $A^S_\zeta=0$.  By Codazzi,
\begin{align*}
        0
        &=
        \left\langle
        (\nabla_X^{\perp,S}A^S)(Y,Z)
        -
        (\nabla_Y^{\perp,S}A^S)(X,Z),
        \zeta
        \right\rangle                                                \\
        &=
        -\left\langle A^S(Y,Z),\nabla_X^{\perp,S}\zeta\right\rangle
        +
        \left\langle A^S(X,Z),\nabla_Y^{\perp,S}\zeta\right\rangle .
\end{align*}
Let $\tau(X)$ be the $E$-part of $-\nabla_X^{\perp,S}\zeta$.  Then
$A^S_{\tau(X)}Y=A^S_{\tau(Y)}X$.  By Lemma \ref{lem:equality-spherical-shape-package},
$X\mapsto A^S_{\tau(X)}$ takes values in $\Herm_0(\mathbb F^n)$.  Lemma
\ref{lem:rigidity-first-prolongation} gives $\tau=0$.  Hence
$\nabla_X^{\perp,S}\zeta\perp E$ for all $X$, so $E$ is parallel.

We next show that $A^S$ is parallel.  Since $E$ is parallel and
$A^S(T\Sigma,T\Sigma)\subset E$, it suffices to test against sections of
$E$.  Fix $x\in\Sigma$ and $\xi\in E_x$, and extend $\xi$ so that
$\nabla^{\perp,S}\xi=0$ at $x$.  Define $C_X^\xi$ by
\begin{align*}
        \langle C_X^\xi Y,Z\rangle
        =
        \left\langle
        (\nabla_X^{\perp,S}A^S)(Y,Z),\xi
        \right\rangle
        \qquad
        \text{at }x .
\end{align*}
Codazzi gives $C_X^\xi Y=C_Y^\xi X$.  Differentiating the Hermitian property of
$A^S_\xi$, using $\nabla J=0$ in the K\"ahler case and
$\nabla\calQ\subset\calQ$ in the quaternionic-K\"ahler case, shows that
$C_X^\xi$ is Hermitian.  Differentiating $\tr A^S_\xi=0$ gives
$\tr C_X^\xi=0$.  Lemma \ref{lem:rigidity-first-prolongation} gives
$C_X^\xi=0$.  Hence
\begin{align}
        \nabla^{\perp,S}A^S=0 .
\label{eq:rigidity-parallel-A}
\end{align}
\end{proof}

%------------------------------------------------------------
\subsection{Identification with the Veronese embeddings}
\label{subsec:rigidity-identification}
%------------------------------------------------------------

Let $W=\Herm_0(T\Sigma;\mathbb F)$ denote the bundle of trace-free $\mathbb F$-Hermitian endomorphisms of
$T\Sigma$. In the K\"ahler case, $\mathbb F = \CC$, and 
\[W_x = \{B\in \End_{\RR}(T_x\Sigma): B^* = B, BJ=JB, \tr_{\CC}(B)=0\}.\]
In the quaternionic-K\"ahler case, $\mathbb F = \HH$, and 
\[W_x = \{B\in \End_{\RR}(T_x\Sigma): B^* = B, B\phi=\phi B\ \text{for all}\ \phi\in\mathcal Q_x, \tr_{\HH}(B)=0\}.\]
We equip $W$ with the fiber metric
\[
        \langle B,C\rangle_W=\tr_{\mathbb F}(BC),
\]
which is the normalization used in the projective models.

\begin{lemma}
\label{lem:curvature-stable-subspaces}
Fix $x\in\Sigma$, and let $V\subset W_x$ be a real linear subspace satisfying
\[
        [R_x^\Sigma(X,Y),V]\subset V
        \qquad
        \text{for every }X,Y\in T_x\Sigma.
\]
Then either $V=\{0\}$ or $V=W_x$.
\end{lemma}

\begin{proof}
Suppose that $V\ne\{0\}$ and choose $0\ne B\in V$. Let
\begin{align*}
	\mathcal U=
	\begin{cases}
		\{\Id,J\}, &\text{in the K\"ahler case,}\\
		\{\Id,I,J,K\}, &\text{in the quaternionic-K\"ahler case.}
	\end{cases}
\end{align*}
For $\phi\in\mathcal U$, we have
\[
        \phi^*=
        \begin{cases}
        \Id,&\phi=\Id,\\
        -\phi,&\phi\ne\Id.
        \end{cases}
\]

Since $B\in W_x$, it is self-adjoint and commutes with every element of
$\mathcal U$. Hence every eigenspace of $B$ is preserved by
$\mathcal U$. We can choose $e_1,\ldots,e_n\in T_x\Sigma$ such that
\[
        \{\psi e_r:\psi\in\mathcal U,\ 1\le r\le n\}
\]
is a real orthonormal basis of $T_x\Sigma$ and $Be_r=\lambda_r e_r$.
Since $B$ commutes with every $\psi\in\mathcal U$, this also gives $B(\psi e_r)=\lambda_r\psi e_r$.

For $1\le r\le n$, let $P_r$ be the orthogonal projection defined by $P_rZ=\sum_{\psi\in\mathcal U}\langle Z,\psi e_r\rangle\,\psi e_r$ and, for $p\ne q$, define $D_{pq}=P_p-P_q\in \End_{\RR}(T_x\Sigma)$. Then we have
\begin{equation}\label{eq:curvature-stable-D-action}
        D_{pq}(\psi e_r)
        =
        \begin{cases}
        \psi e_p,&r=p,\\
        -\psi e_q,&r=q,\\
        0,&r\notin\{p,q\}.
        \end{cases}
\end{equation}
For $p\ne q$ and $\phi\in\mathcal U$, we also define
\[
        \mathcal K_{pq}^{\phi}
        =
        \frac{4}{K}
        R_x^\Sigma(e_p,\phi e_q).
\]
By Proposition \ref{lem:equality-intrinsic-space-form}, the formulae \eqref{eq:complex-curvature} and
\eqref{eq:quaternionic-curvature} give
\begin{equation}\label{eq:curvature-stable-K-wedge}
        \mathcal K_{pq}^{\phi}
        =
        \sum_{\psi\in\mathcal U}
        \psi e_p\wedge\psi\phi e_q.
\end{equation}
This gives
\begin{equation}\label{eq:curvature-stable-K-action}
\begin{aligned}
        \mathcal K_{pq}^{\phi}(\psi e_p)
        &=\psi\phi e_q,\\
        \mathcal K_{pq}^{\phi}(\psi e_q)
        &=-\psi\phi^*e_p,\\
        \mathcal K_{pq}^{\phi}(\psi e_r)
        &=0
        \qquad(r\notin\{p,q\}).
\end{aligned}
\end{equation}

We next consider the commutators
\[
        H_{pq}^{\phi}
        =
        \frac12[\mathcal K_{pq}^{\phi},D_{pq}].
\]
Using \eqref{eq:curvature-stable-D-action} and
\eqref{eq:curvature-stable-K-action}, we obtain
\begin{equation}\label{eq:curvature-stable-H-action}
\begin{aligned}
        H_{pq}^{\phi}(\psi e_p)
        &=\psi\phi e_q,\\
        H_{pq}^{\phi}(\psi e_q)
        &=\psi\phi^*e_p,\\
        H_{pq}^{\phi}(\psi e_r)
        &=0
        \qquad(r\notin\{p,q\}).
\end{aligned}
\end{equation}
In particular, $D_{pq}$ and
$H_{pq}^{\phi}$ commute with every element of $\mathcal U$.

We then observe from \eqref{eq:curvature-stable-D-action}, \eqref{eq:curvature-stable-K-action} and
\eqref{eq:curvature-stable-H-action} that
\begin{equation}\label{eq:curvature-stable-K-D-H}
        [\mathcal K_{pq}^{\phi},D_{pq}]
        =
        2H_{pq}^{\phi},
        \qquad
        [\mathcal K_{pq}^{\Id},H_{pq}^{\Id}]
        =
        -2D_{pq}.
\end{equation}
On the other hand, since $B\in V$ is trace-free, we can find $i\neq j$ such that $\lambda_i\neq \lambda_j$. Then the assumption implies 
\[
        H_{ij}^{\Id}
        =
        \frac{1}{\lambda_i-\lambda_j}
        [\mathcal K_{ij}^{\Id},B]
        \in V.
\]
Hence by \eqref{eq:curvature-stable-K-D-H} and
the assumption, we obtain
\[
        D_{ij}
        =
        -\frac12
        [\mathcal K_{ij}^{\Id},H_{ij}^{\Id}]
        \in V.
\]

Suppose that $q\notin\{i,j\}$. From \eqref{eq:curvature-stable-D-action}, \eqref{eq:curvature-stable-K-action},
\eqref{eq:curvature-stable-H-action} and the assumption we also have
\[
        [\mathcal K_{iq}^{\Id},D_{ij}]
        =
        H_{iq}^{\Id}
        \in V.
\]
Equation \eqref{eq:curvature-stable-K-D-H} and the assumption then gives
\[
        D_{iq}
        =
        -\frac12
        [\mathcal K_{iq}^{\Id},H_{iq}^{\Id}]
        \in V.
\]
Since $D_{pq}=D_{iq}-D_{ip}$ for $p,q\neq i$, we conclude that
\begin{equation}\label{eq:curvature-stable-all-D}
        D_{pq}\in V.
\end{equation}
for every $p\neq q$. Finally, \eqref{eq:curvature-stable-K-D-H} and the assumption imply
\begin{equation}\label{eq:curvature-stable-all-H}
        H_{pq}^{\phi}
        =
        \frac12[\mathcal K_{pq}^{\phi},D_{pq}]
        \in V.
\end{equation}
for every $p\neq q$ and $\phi\in\mathcal U$.

We next show that the endomorphisms $D_{pq}$ and
$H_{pq}^{\phi}$ generate every element of
$W_x$. Let $C\in W_x$ be arbitrary. Define
\[
        a_r=\langle Ce_r,e_r\rangle,
        \qquad
        c_{pq}^{\phi}
        =
        \langle Ce_p,\phi e_q\rangle\quad
        (p<q,\, \phi\in\mathcal U).        
\]
Since $C$ is trace-free,
\begin{equation}\label{eq:curvature-stable-sum-a}
        \sum_{r=1}^na_r=0.
\end{equation}

If $\phi\ne\Id$, then self-adjointness of $C$, the property
$C\phi=\phi C$, and $\phi^*=-\phi$ give
\begin{equation}\label{eq:curvature-stable-diagonal-zero}
        \langle Ce_r,\phi e_r\rangle=0
        \qquad(\phi\ne\Id).
\end{equation}
For $p<q$ and $\phi\in\mathcal U$, the same properties give
\begin{align}
        \langle Ce_q,\phi^*e_p\rangle
        =
        c_{pq}^{\phi}.
\label{eq:curvature-stable-opposite-coefficient}
\end{align}

Expanding $Ce_r$ in the chosen real orthonormal basis and using
\eqref{eq:curvature-stable-diagonal-zero} and
\eqref{eq:curvature-stable-opposite-coefficient}, we obtain
\begin{equation}\label{eq:curvature-stable-Cer}
        Ce_r
        =
        a_re_r
        +
        \sum_{q=r+1}^n
        \sum_{\phi\in\mathcal U}
        c_{rq}^{\phi}\,\phi e_q
        +
        \sum_{p=1}^{r-1}
        \sum_{\phi\in\mathcal U}
        c_{pr}^{\phi}\,\phi^*e_p.
\end{equation}

On the other hand, consider the endomorphism
\[
        C_0
        =
        \sum_{r=1}^{n-1}a_rD_{rn}
        +
        \sum_{1\le p<q\le n}
        \sum_{\phi\in\mathcal U}
        c_{pq}^{\phi}H_{pq}^{\phi}.
\]
Using \eqref{eq:curvature-stable-D-action},
\eqref{eq:curvature-stable-H-action}, and
\eqref{eq:curvature-stable-sum-a}, we find that
\[
        C_0e_r
        =
        a_re_r
        +
        \sum_{q=r+1}^n
        \sum_{\phi\in\mathcal U}
        c_{rq}^{\phi}\,\phi e_q
        +
        \sum_{p=1}^{r-1}
        \sum_{\phi\in\mathcal U}
        c_{pr}^{\phi}\,\phi^*e_p.
\]
Comparison with \eqref{eq:curvature-stable-Cer} gives $C_0e_r=Ce_r$ for every $1\leq r\leq n$.
Furthermore, both $C_0$ and $C$ commute with every element of
$\mathcal U$. Hence $C_0(\psi e_r) = C(\psi e_r)$.
Thus $C_0$ and $C$ agree on a real basis of $T_x\Sigma$, and therefore $C=C_0$.
We have proved the exact decomposition
\begin{equation}\label{eq:curvature-stable-decomposition}
        C
        =
        \sum_{r=1}^{n-1}a_rD_{rn}
        +
        \sum_{1\le p<q\le n}
        \sum_{\phi\in\mathcal U}
        c_{pq}^{\phi}H_{pq}^{\phi}.
\end{equation}

By \eqref{eq:curvature-stable-all-D} and
\eqref{eq:curvature-stable-all-H}, every term on the right-hand side of
\eqref{eq:curvature-stable-decomposition} belongs to $V$. Hence
$C\in V$. Since $C\in W_x$ was arbitrary, $W_x\subset V$, and the assertion follows.
\end{proof}

\begin{remark}
The preceding proof is the elementary pointwise form of a standard holonomy
irreducibility fact.  By Proposition~\ref{lem:equality-intrinsic-space-form}, $(\Sigma,g)$ is isometric to $(\mathbb FP^2, g_{FS})$, the curvature endomorphisms
generate $\mathfrak u(n)$, respectively
$\mathfrak{sp}(n)\oplus\mathfrak{sp}(1)$.  The action on
$W_x=\Herm_0(T_x\Sigma;\mathbb F)$ is the commutator action.

In the K\"ahler case, the central $\mathfrak u(1)$ acts trivially, and
$B\mapsto iB$ identifies $\Herm_0(\mathbb C^n)$ with $\mathfrak{su}(n)$,
where the action is the adjoint action.  Since $\mathfrak{su}(n)$ is simple,
this real representation is irreducible.  In the quaternionic case, the
$\mathfrak{sp}(1)$-factor acts trivially, while the
$\mathfrak{sp}(n)$-module $\Herm_0(\mathbb H^n)$ is the isotropy
representation of the irreducible symmetric space $SU^*(2n)/Sp(n)$ and is therefore irreducible.  Thus the lemma also follows from standard
holonomy irreducibility; see \cite{AmbroseSinger1953} and
\cite[Chs.~V--VI, X]{Helgason1978}.
\end{remark}

\vspace{0.5cm}

By Lemma~\ref{lem:equality-spherical-shape-package}, $A^S_\eta\in W_x$ for every $\eta\in N_x^S\Sigma$. Therefore the shape operators define a bundle map
\[
        L:E\longrightarrow W,
        \qquad
        L(\eta)=A^S_\eta .
\]

\begin{proposition}
\label{prop:rigidity-normal-form}
The map $L$ is a parallel bundle isomorphism and
\begin{equation}\label{eq:rigidity-shape-map-homothety}
\sqrt{\frac{n+1}{n}}\,L:E\longrightarrow W
\end{equation}
is a parallel bundle isometry. 
\end{proposition}

\begin{proof}
Since $E$ is parallel and $\nabla^{\perp,S}A^S=0$ by
Lemma~\ref{lem:parallel-E}, the map $L$ is parallel. Hence
\[
        T:=LL^*:W\longrightarrow W
\]
is a parallel self-adjoint endomorphism. 	

Fix $x\in\Sigma$. Then $T_x\colon W_x\to W_x$  has an eigenvalue $\lambda$ and a corresponding nonzero eigenspace
\[
        V_\lambda=\ker(T_x-\lambda\Id_{W_x}).
\]
Since $T$ is parallel, we have
\[
        T_x\bigl([R_x^\Sigma(X,Y),B]\bigr)
        =
        [R_x^\Sigma(X,Y),T_xB]
\]
for every $X,Y\in T_x\Sigma$ and $B\in W_x$. Hence $[R_x^\Sigma(X,Y),V_\lambda]\subset V_\lambda$.
Lemma~\ref{lem:curvature-stable-subspaces} therefore gives $V_\lambda=W_x$. Consequently, $T_x=\lambda\Id_{W_x}$ for any $x\in\Sigma$. Since $T$ is parallel and $\Sigma$ is connected, we have $T=\lambda\Id_W$ for some constant $\lambda\in\RR$.

We now determine the constant $\lambda$.  Choose an orthonormal basis $\eta_1,\ldots,\eta_r$ of $E_x$, and choose an $\mathbb F$-linear isometry $T_x\Sigma\simeq \mathbb F^n$.
For a unit vector $v\in T_x\Sigma\simeq\mathbb F^n$, set
\[
        P_v=vv^*-\frac1n\Id .
\]
Then $P_v\in W_x$ and
\[
        |P_v|^2
        =
        \tr_{\mathbb F}
        \left(vv^*-\frac1n\Id\right)^2
        =
        1-\frac2n+\frac1n
        =
        \frac{n-1}{n}.
\]
Since $A^S_{\eta_a}$ is trace-free, $\langle A^S_{\eta_a}v,v\rangle =  \langle A^S_{\eta_a},P_v\rangle$.
Therefore, Lemma~\ref{lem:equality-spherical-shape-package} gives
\begin{align*}
        \langle T_xP_v,P_v\rangle
        =
        |L_x^*P_v|^2                                                         
        =
        \sum_{a=1}^r
        \langle A^S_{\eta_a},P_v\rangle^2                                    
        =
        |A^S(v,v)|^2
        =
        \frac{n-1}{n+1}.
\end{align*}
On the other hand, since $T=\lambda\Id_W$,
\[
        \langle T_xP_v,P_v\rangle
        =
        \lambda |P_v|^2
        =
        \lambda\frac{n-1}{n}.
\]
Combining the above two identities, we obtain $\lambda=\frac n{n+1}$.
Therefore
\begin{equation}\label{eq:rigidity-A-Casimir}
        LL^*
        =
        \frac n{n+1}\Id_W .
\end{equation}

Consequently, $L_x:E_x\to W_x$ is surjective.  Indeed, for every
$B\in W_x$, $B=\frac{n+1}{n}L_xL_x^*B\in \operatorname{im}L_x$.
Moreover, $L_x$ is injective.  If $\eta\in E_x$ and $L_x\eta=0$, then
$A^S_\eta=0$, so $\langle A^S(X,Y),\eta\rangle =  \langle A^S_\eta X,Y\rangle = 0$ for all $X,Y\in T_x\Sigma$. 
Since
\[
        E_x=\Span\{A^S(X,Y):X,Y\in T_x\Sigma\},
\]
this forces $\eta=0$.  Therefore $L_x:E_x\to W_x$ is an isomorphism for every $x\in\Sigma$.
Since $L$ is an isomorphism and $LL^*=\frac n{n+1}\Id_W$, we also have $L^*L=\frac n{n+1}\Id_E$. Thus the assertion follows.
\end{proof}

\begin{lemma}
\label{lem:reduction-of-codimension}
There exists a linear subspace $V_0\subset\RR^{m+\ell}$ such that
\[
        F(\Sigma)\subset S(V_0)
        :=
        V_0\cap S^{m+\ell-1}(1),
\]
and the normal bundle of $F:\Sigma\longrightarrow S(V_0)$ is precisely $E$.
\end{lemma}

\begin{proof}
For $x\in\Sigma$, set
\[
        V_x
        =
        \mathbb RF(x)\oplus T_x\Sigma\oplus E_x
        \subset\RR^{m+\ell},
\]
where $T_x\Sigma$ is identified with its image under $dF_x$, and define
\[
        \mathcal V
        =
        \bigcup_{x\in\Sigma}\{x\}\times V_x
        \subset
        \Sigma\times\RR^{m+\ell}.
\]
Since the three summands are mutually orthogonal and have constant ranks,
$\mathcal V$ is a smooth Euclidean subbundle.

Let $D$ denote the standard flat connection on the trivial bundle
$\Sigma\times\RR^{m+\ell}$. If $s=aF+Y+\eta$ is a local section of $\mathcal V$, where
$Y\in\Gamma(T\Sigma)$ and $\eta\in\Gamma(E)$, then the Euclidean
Gauss--Weingarten formulas give
\[
\begin{aligned}
        D_Xs
        =
        \bigl(X(a)-\langle X,Y\rangle\bigr)F
        +
        \bigl(aX+\nabla_XY-A^S_\eta X\bigr)
        +
        \bigl(A^S(X,Y)+\nabla_X^{\perp,S}\eta\bigr).
\end{aligned}
\]
Since $A^S(T\Sigma,T\Sigma)\subset E$ and $E$ is parallel in the spherical normal connection, the three lines
on the right belong respectively to $\mathbb RF,\  T\Sigma,\  E$. Consequently, $D_X\Gamma(\mathcal V)\subset\Gamma(\mathcal V)$.

Let $\Pi_x:\RR^{m+\ell}\longrightarrow V_x$ be the orthogonal projection. Since $D$ is metric and preserves
$\mathcal V$, it also preserves $\mathcal V^\perp$. Hence the induced
connection on $\Sigma\times\End(\RR^{m+\ell})$ satisfies $D^{\End}\Pi=0$.
Under the standard trivialization of the endomorphism bundle,
$D^{\End}\Pi=0$ means that $d\Pi=0$ for the map $\Pi:\Sigma\longrightarrow\End(\RR^{m+\ell})$.
Since $\Sigma$ is connected, $\Pi$ is constant. 

Fix $x_0\in\Sigma$ and set
\[
        V_0=\operatorname{Im}\Pi_{x_0}.
\]
Then $V_x=\operatorname{Im}\Pi_x=V_0$ for every $x\in\Sigma$. Since $ F(x)\in\mathbb RF(x)\subset V_x=V_0$, we obtain
\[
        F(\Sigma)\subset S(V_0).
\]
Furthermore,
\[
\begin{aligned}
        T_{F(x)}S(V_0)
        =
        V_0\cap F(x)^\perp
        =
        T_x\Sigma\oplus E_x.
\end{aligned}
\]
Therefore the orthogonal complement of $T_x\Sigma$ in
$T_{F(x)}S(V_0)$ is $E_x$. Thus $E$ is the full normal bundle
of $F:\Sigma\to S(V_0)$.
\end{proof}

By Lemma~\ref{lem:reduction-of-codimension}, we henceforth regard $F$ as
an isometric immersion
\[
        F:\Sigma\longrightarrow S(V_0)
\]
whose full spherical normal bundle is $E$.  This reduction will be used
below when we compare the full normal bundle of the lifted immersion with
that of the standard model.

We next identify $F$ with the standard Veronese embeddings.  Let
\[
        M_{\mathbb F}
        =
        (\mathbb FP^n,g_{\mathbb F}),
        \qquad
        g_{\mathbb F}
        =
        \Phi_{\mathbb F}^*g_{S(\mathcal H_{\mathbb F})},
\]
which is the standard Veronese embedding given by \eqref{eq:model-FPn-map}.
Proposition~\ref{lem:equality-intrinsic-space-form} gives a global isometry
\begin{equation}\label{eq:rigidity-universal-cover-isometry}
        \varphi:
        \Sigma
        \longrightarrow
        M_{\mathbb F}
\end{equation}
such that $d\varphi$ preserves the complex/quaternionic structure.

We now compare the immersion with the standard Veronese
embedding on the common base $\Sigma$.  Set
\[
        F_\Phi
        =
        \Phi_{\mathbb F}\circ\varphi:
        \Sigma\longrightarrow S(\mathcal H_{\mathbb F}).
\]
Let $\overline E$ be the spherical normal bundle of the Veronese embedding
$\Phi_{\mathbb F}$, and let
\[
        \overline L:\overline E\longrightarrow\overline W,
        \qquad
        \overline L(\overline\eta)=\overline A^S_{\overline\eta},
        \qquad
        \overline W=\Herm_0(TM_{\mathbb F};\mathbb F),
\]
be its shape-operator map.  Pull back the model bundles to $\Sigma$:
\[
        E_\Phi=\varphi^*\overline E,
        \qquad
        W_\Phi=\varphi^*\overline W,
        \qquad
        L_\Phi=\varphi^*\overline L:
        E_\Phi\longrightarrow W_\Phi .
\]
Thus $E$ and $E_\Phi$ are the full spherical normal bundles of $F$ and $F_\Phi$, respectively.

\begin{lemma}
\label{lem:rigidity-lifted-normal-isometry}
There exists a parallel bundle isometry
\[
\Psi:E\longrightarrow E_\Phi
\]
such that
\begin{equation}\label{eq:rigidity-second-fundamental-form-matched}
        \Psi\bigl(A^{S,F}(X,Y)\bigr)
        =
        A^{S,F_\Phi}(X,Y)
\end{equation}
for all $X,Y\in T\Sigma$.
\end{lemma}

\begin{proof}
From Proposition~\ref{lem:equality-intrinsic-space-form}, the differential $d\varphi$ preserves the relevant $\mathbb F$-structure,
and therefore identifies the pulled-back endomorphism bundle $W_\Phi$ with
$W$.  Explicitly, we can define a bundle map
\[
        J_\varphi:W_\Phi\longrightarrow W,
        \qquad
        (J_\varphi C)_x
        =
        (d\varphi_x)^{-1}\circ C_x\circ d\varphi_x .
\]
This is an isometric bundle isomorphism.  It is also parallel.  Indeed,
$\varphi$ is a local isometry, so
\[
        d\varphi(\nabla^\Sigma_XY)
        =
        \nabla^{M_{\mathbb F}}_{d\varphi X}(d\varphi Y),
\]
and hence conjugation by $d\varphi$ intertwines the induced connections on
the trace-free Hermitian endomorphism bundles:
\[
        \nabla_X^{W}(J_\varphi C)
        =
        J_\varphi(\nabla_X^{W_\Phi}C).
\]

Now set
\[
        \widehat L_\Phi=J_\varphi\circ L_\Phi:
        E_\Phi\longrightarrow W .
\]
By Proposition~\ref{prop:rigidity-normal-form},
$\sqrt{(n+1)/n}\,L$ is a parallel bundle isometry.  For the standard
projective embedding, the explicit formula \eqref{eq:model-FPn-spherical-A},
together with equivariance, gives that
$\sqrt{(n+1)/n}\,\overline L$ is a parallel bundle isometry; equivalently,
$\sqrt{(n+1)/n}\,L_\Phi$ is a parallel bundle isometry.  Since $J_\varphi$
is a parallel isometry, both $\sqrt{(n+1)/n}\, L$ and
$\sqrt{(n+1)/n}\,\widehat L_\Phi$ are parallel bundle isometries onto the
same bundle $W$.

We then define
\begin{equation}\label{eq:rigidity-normal-bundle-isometry}
        \Psi
        =
        \widehat L_\Phi^{-1}\circ L:
        E\longrightarrow E_\Phi .
\end{equation}
Because $L$ and $\widehat L_\Phi$ are parallel homotheties with the
same factor, $\Psi$ is a parallel bundle isometry.

Finally, $L$ and $\widehat L_\Phi$ are precisely the shape-operator
maps of the two immersions, viewed as endomorphisms of $T\Sigma$:
\[
        L_x(\eta)=A^{S,F}_\eta,
        \qquad
        (\widehat L_\Phi)_x(\zeta)=A^{S,F_\Phi}_\zeta .
\]
Therefore
\[
        A^{S,F_\Phi}_{\Psi\eta}
        =
        \widehat L_\Phi(\Psi\eta)
        =
        L(\eta)
        =
        A^{S,F}_{\eta}.
\]
Since $\Psi$ is onto, \eqref{eq:rigidity-second-fundamental-form-matched}
then follows by pairing with arbitrary normal vectors: for every $\eta\in E$,
\[
\begin{aligned}
        \left\langle
        \Psi A^{S,F}(X,Y),\Psi\eta
        \right\rangle
        &=
        \left\langle
        A^{S,F}(X,Y),\eta
        \right\rangle                                      \\
        &=
        \left\langle
        A^{S,F}_\eta X,Y
        \right\rangle                                      \\
        &=
        \left\langle
        A^{S,F_\Phi}_{\Psi\eta}X,Y
        \right\rangle                                      \\
        &=
        \left\langle
        A^{S,F_\Phi}(X,Y),\Psi\eta
        \right\rangle .
\end{aligned}
\]
\end{proof}

\begin{proposition}
\label{prop:rigidity-ambient-congruence}
There exists an orthogonal linear isomorphism
\[
        Q:V_0\longrightarrow\mathcal H_{\mathbb F}
\]
such that
\begin{equation}\label{eq:rigidity-global-congruence}
        Q\circ F
        =
        F_\Phi
        =
        \Phi_{\mathbb F}\circ\varphi .
\end{equation}
\end{proposition}

\begin{proof}
For each $x\in\Sigma$, we have the bundle decompositions
\[
\begin{aligned}
        V_0
        &=
        \mathbb R F(x)
        \oplus
        dF_x(T_x\Sigma)
        \oplus
        E_x,                                      \\
        \mathcal H_{\mathbb F}
        &=
        \mathbb R F_\Phi(x)
        \oplus
        d(F_\Phi)_x(T_x\Sigma)
        \oplus
        (E_\Phi)_x .
\end{aligned}
\]
Define $Q_x:V_0\to\mathcal H_{\mathbb F}$ by
\[
        Q_x\bigl(
        aF(x)+dF_xX+\eta
        \bigr)
        =
        aF_\Phi(x)+d(F_\Phi)_xX+\Psi_x\eta .
\]
The two immersions induce the same metric on $\Sigma$, and $\Psi_x$ is an
isometry.  Hence $Q_x$ maps the three mutually orthogonal summands for $F$
isometrically onto the corresponding summands for $F_\Phi$.  Thus each $Q_x$
is an orthogonal linear isomorphism.

It remains to show that $Q_x$ is independent of $x$.  Let $D^{V_0}$ and
$D^{\mathcal H}$ be the flat connections on the trivial bundles
$\Sigma\times V_0$ and $\Sigma\times\mathcal H_{\mathbb F}$, and let
$Q_\bullet$ denote the bundle map with fiber map $Q_x$ at $x$.  Every local
section of $\Sigma\times V_0$ can be written uniquely as
\[
        s=aF+dF(Y)+\eta,
        \qquad
        \eta\in\Gamma(E).
\]
The spherical Gauss--Weingarten formulas give
\[
\begin{aligned}
        D^{V_0}_Xs
        ={}&
        \bigl(X(a)-\langle X,Y\rangle\bigr)F       \\
        &+
        dF\bigl(
        aX+\nabla^\Sigma_XY-A^{S,F}_\eta X
        \bigr)                                                \\
        &+
        A^{S,F}(X,Y)
        +
        \nabla_X^{\perp,S,F}\eta .
\end{aligned}
\]
Using the same formula for $F_\Phi$, together with the parallelism of $\Psi$
and \eqref{eq:rigidity-second-fundamental-form-matched}, we obtain
\[
\begin{aligned}
        D^{\mathcal H}_X(Q_\bullet s)
        ={}&
        \bigl(X(a)-\langle X,Y\rangle\bigr)F_\Phi       \\
        &+
        dF_\Phi\bigl(
        aX+\nabla^\Sigma_XY-A^{S,F}_\eta X
        \bigr)                                                    \\
        &+
        \Psi\bigl(
        A^{S,F}(X,Y)
        +
        \nabla_X^{\perp,S,F}\eta
        \bigr)                                                    \\
        ={}&
        Q_\bullet(D^{V_0}_Xs).
\end{aligned}
\]
Thus $Q_\bullet$ intertwines the two flat connections.  If $v\in V_0$ is
regarded as a constant section of $\Sigma\times V_0$, then
\[
        D^{\mathcal H}_X(Q_\bullet v)
        =
        Q_\bullet(D^{V_0}_Xv)
        =
        0.
\]
Hence $x\mapsto Q_xv$ is constant.  Since $\Sigma$ is connected and this
holds for every $v\in V_0$, all maps $Q_x$ are equal.  Denote their common
value by $Q$.  Then $Q$ is an orthogonal linear isomorphism, and by
construction
\[
        Q\circ F=F_\Phi=\Phi_{\mathbb F}\circ\varphi .
\]
This is exactly \eqref{eq:rigidity-global-congruence}.
\end{proof}

\begin{proof}[Completion of proof of Theorem \ref{thm:main-ball-FPn}]
By Proposition~\ref{prop:rigidity-ambient-congruence}, there exists an
orthogonal linear isomorphism $Q:V_0\longrightarrow\mathcal H_{\mathbb F}$ such that
\[
        Q\circ F
        =
        \Phi_{\mathbb F}\circ\varphi .
\]
Therefore $F$ is globally congruent to the first standard projective
embedding, up to the totally geodesic inclusion
$S(V_0)\subset S^{m+\ell-1}(1)$.  This completes the proof of
Theorem~\ref{thm:main-ball-FPn}.
\end{proof}
\bigskip

%============================================================

\end{document}